\documentclass{article}

\usepackage{microtype}
\usepackage[titletoc]{appendix}

\usepackage{graphicx,color}
\usepackage{amsmath,amsfonts,amsthm,bm,amssymb}
\usepackage[a4paper]{geometry}

\newcommand{\Z}{\mathbb{Z}} 
\newcommand{\C}{\mathbb{C}} 
\newcommand{\N}{\mathbb{N}} 
\newcommand{\M}{\mathbb{M}}

\newcommand{\J}{\mathcal{J}} 

\newcommand{\imagunit}{\mathrm{i}}
\newcommand{\twopii}{2 \pi \imagunit \,}

\newcommand{\shinv}{\mathrm{shinv}}

\DeclareSymbolFont{bbold}{U}{bbold}{m}{n}
\DeclareSymbolFontAlphabet{\mathbbold}{bbold}
\newcommand{\ind}[1]{\mathbbold{1}_{#1}}

\newcommand*{\link}[1]{\eqref{#1}}                                      
\newcommand*{\abs}[1]{\left| #1 \right|}                                
\newcommand*{\nach}{\rightarrow}                                        
\newcommand*{\norm}[1]{\left\| #1 \right\|}                             
\newcommand*{\sep}{\; \vrule \;}                                       
\newcommand{\ie}{i.e., }
\renewcommand*{\S}{\mathcal{S}}                                         
\newcommand*{\distr}[2]{\left\langle #1, #2 \right\rangle}              

\newcommand*{\SI}{\mathfrak{S}}                                         
\newcommand*{\id}{\mathrm{id}}                                          
\newcommand*{\0}{\mathcal{O}}                                           

\usepackage[draft=false,colorlinks=true,bookmarksnumbered,linkcolor=black, citecolor=black, urlcolor=black]{hyperref}

\usepackage{aliascnt}

\theoremstyle{plain}
\newtheorem{theorem}{Theorem}[section]

\newaliascnt{lem}{theorem}
\newtheorem{lemma}[lem]{Lemma}

\aliascntresetthe{lem}
\newaliascnt{prop}{theorem}
\newtheorem{proposition}[prop]{Proposition}

\aliascntresetthe{prop}
\newaliascnt{cor}{theorem}
\newtheorem{corollary}[cor]{Corollary}

\aliascntresetthe{cor}

\theoremstyle{definition}

\newaliascnt{ex}{theorem}
\newtheorem{examplex}[ex]{Example}

\aliascntresetthe{ex}
\newenvironment{example}
  {\pushQED{\qed}\examplex}
  {\popQED\endexamplex}

\theoremstyle{remark}

\newaliascnt{rem}{theorem}
\newtheorem{remarkx}[rem]{Remark}

\aliascntresetthe{rem}
\newenvironment{remark}
  {\pushQED{\qed}\remarkx}
  {\popQED\endremarkx}

\newcommand{\setu}{\mathfrak{u}}

\newcommand{\rd}{\,\mathrm{d}} 

\newcommand{\bszero}{\boldsymbol{0}} 
\newcommand{\bst}{\boldsymbol{t}}    
\newcommand{\bsh}{\boldsymbol{h}}    
\newcommand{\bsk}{\boldsymbol{k}}    
\newcommand{\bsx}{\boldsymbol{x}}    
\newcommand{\bsy}{\boldsymbol{y}}    
\newcommand{\bsz}{\boldsymbol{z}}    
\newcommand{\bsDelta}{\boldsymbol{\Delta}}    

\newcommand{\NR}[1]{\mu_R(#1)}
\newcommand{\NN}{\eta}

\newcommand{\tpmod}[1]{~(\operatorname{mod}{#1})} 

\def\citep#1#2{\cite[{#1}]{#2}}

\newcommand{\RefEq}[1]{~\textup{(\ref{#1})}}

\newcommand{\RefSec}[1]{Section~\textup{\ref{#1}}}
\newcommand{\RefThm}[1]{Theorem~\textup{\ref{#1}}}
\newcommand{\RefProp}[1]{Proposition~\textup{\ref{#1}}}
\newcommand{\RefLem}[1]{Lemma~\textup{\ref{#1}}}
\newcommand{\RefCol}[1]{Corollary~\textup{\ref{#1}}}
\newcommand{\RefEx}[1]{Example~\textup{\ref{#1}}}
\newcommand{\RefRem}[1]{Remark~\textup{\ref{#1}}}

\begin{document}

\title{Rank-$1$ lattice rules for multivariate integration \\ in spaces of permutation-invariant functions: \\
Error bounds and tractability\footnote{Preprint of the corresponding article published within \emph{Advances in Computational Mathematics}. The final publication is available at \href{http://link.springer.com}{link.springer.com}, see \href{http://dx.doi.org/10.1007/s10444-015-9411-6}{DOI:10.1007/s10444-015-9411-6}.}}
\author{
Dirk Nuyens\thanks{Department of Computer Science, KU Leuven, Celestijnenlaan 200A, 3001 Heverlee, Belgium.}~\thanks{Email: dirk.nuyens@cs.kuleuven.be.}
\and Gowri Suryanarayana\footnotemark[2]~\thanks{Email: gowri.suryanarayana@cs.kuleuven.be}
\and Markus Weimar\thanks{Corresponding author. Philipps-University Marburg, Faculty of Mathematics and Computer Science, Workgroup Numerics and Optimization, Hans-Meerwein-Stra{\ss}e, Lahnberge, 35032 Marburg, Germany. Email: weimar@mathematik.uni-marburg.de.}
}

\maketitle

\begin{abstract}
\noindent We study multivariate integration of functions that are invariant under permutations (of subsets) of their arguments.
We find an upper bound for the $n$th minimal worst case error and show that under certain conditions, it can be bounded independent of the number of dimensions. In particular, we study the application of unshifted and 
randomly shifted rank-$1$ lattice rules in such a problem setting.
We derive conditions under which multivariate integration is polynomially or strongly polynomially tractable with
the Monte Carlo rate of convergence $\0(n^{-1/2})$. Furthermore, we prove that those tractability results can be achieved with shifted lattice rules and that the shifts are indeed necessary. Finally, we show the existence of rank-$1$ lattice rules whose worst case error on the permutation- and shift-invariant spaces converge with (almost) optimal rate.
That is, we derive error bounds of the form $\0(n^{-\lambda/2})$ for all $1 \leq \lambda < 2\alpha$, where $\alpha$ denotes the smoothness of the spaces.

\smallskip
\noindent \textbf{Keywords:} \textit{Numerical integration, Quadrature, Cubature, Quasi-Monte Carlo methods, Rank-$1$ lattice rules. 
}

\smallskip
\noindent \textbf{Subject Classification:} 65D30, 
65D32, 
65C05,  
65Y20, 
68Q25,  
68W40.
\end{abstract}

\section{Introduction and main results}\label{sect:Intro}
The approximation of multivariate integrals is a very old and popular topic of research.
In modern science the \emph{efficient} numerical treatment of very high-dimensional integration problems becomes
more and more important. Therefore one seeks for algorithms which satisfy error bounds with a higher-order rate of convergence and a moderate dependence on the dimension at the same time.
By now it is well-known that, when working with a huge number of dimensions, some additional a priori knowledge on the integrands under consideration is needed in order to reduce the information complexity and thus the computational hardness of such problems.
Usually this additional knowledge is modeled by the use of function spaces endowed with weighted norms that allow to control the influence of different (groups of) variables on the functions one likes to integrate; see \cite{DKS13} for a survey.
Another kind of additional knowledge, given in terms of permutation-invariance conditions, was proposed recently; see \cite{W12,W14}.
 In this paper we exploit such conditions in order to bound the worst case error of general cubature methods 
 for the integration of periodic functions defined on the $d$-dimensional unit cube, where $d\in\N$ can be arbitrary large. Besides proving the existence of good quasi-Monte Carlo (QMC) algorithms 
 based on well-known averaging techniques we focus on shifted and unshifted rank-$1$ lattice rules.
In contrast to Monte Carlo algorithms which use $n$ independent random samples those integration methods are based on very structured, deterministic point sets.
Our setting is motivated by problems from computational quantum physics. Recently it has been shown that the rate of convergence for solving the electronic Schr\"{o}dinger equation of $N$ electron systems does not depend on the number of electrons \cite{Y10}. This is due to the intrinsic property of the system that electronic wavefunctions are antisymmetric with respect to the exchange of electrons with the same spin. We observe from \cite{W12} that finding the approximate solution to such problems involves the calculation of inner products of two antisymmetric functions, i.e., the integration of permutation-invariant functions.

We now briefly describe our main results and the organization of the material.
To begin with, in \autoref{sect:setting} we present the setting we are going to study. Here we introduce the reproducing kernel Hilbert spaces (RKHSs), as well as their permutation-invariant subspaces, our integrands come from. We recall the definition of (weighted) cubature rules and their worst case errors. Finally, we briefly review some well-known concepts from information-based complexity.
\autoref{sect:avgbounds} then deals with existence results obtained by averaging. 
In particular, in \autoref{thm:tractability} we prove that there are (equal weight) QMC rules which satisfy error bounds that decay with the Monte Carlo rate of convergence $\0(n^{-1/2})$ while the implied constant grows only polynomially with the dimension~$d$ provided there is sufficient permutation-invariance. 
Under fairly moderate assumptions on the underlying function space, these error bounds do not depend on $d$ at all. 
That is, e.g., for the fully permutation-invariant problem we prove strong polynomial tractability (e.g., in periodic Sobolev spaces). 
We contrast our results with well-known tractability assertions for related integration problems defined on weighted spaces.
Finally, \autoref{sect:lattice_rules} is devoted to the study of rank-$1$ lattice rules. It contains our main results. In \autoref{subsect:unshifted_lattice rules} we start by proving exact error formulas for unshifted rules which imply lower  bounds showing that no such rule can attain the generic upper bounds stated in \autoref{thm:tractability}. Consequently (independently of the problem parameters) this class of algorithms is too small to obtain strong polynomial tractability.
Therefore, in \autoref{subsect:shifted_lattice_rules}, we turn to (randomly) shifted rank-$1$ lattice rules which are related to certain permutation- and shift-invariant RKHSs. We derive exact expressions for the associated kernels and for the root mean squared worst case error $E(Q_n(\bsz))$ (w.r.t.\ the random shift) of the integration algorithms under consideration. 
These formulas then lead us to lower bounds for $E(Q_n(\bsz))$ and to the observation that shifted rules outperform their unshifted counterparts. Finally, our main result (\autoref{thm:existence}) states that there exist generating vectors $\bsz^*$ such that (on average) the error of the shifted rank-$1$ lattice rule $Q_n(\bsz^*)+\bm{\Delta}$ is bounded by $\0(n^{-\lambda/2})$, where $\lambda/2$ can be chosen arbitrarily close to $\alpha$ (the smoothness parameter of the space under consideration and the optimal rate of convergence for  these rules). 
For $\lambda=1$ the bounds proven in \autoref{thm:existence} resemble the generic upper bounds given in \autoref{sect:avgbounds}. 
Hence, under suitable conditions shifted rank-$1$ lattice rules imply strong polynomial tractability for the integration of permutation-invariant functions.
In \autoref{sect:appendix} we conclude the paper with an appendix which contains the proofs of some technical lemmas needed for our derivation.


\section{Setting}\label{sect:setting}
\subsection{Subspaces of permutation-invariant functions}\label{subsect:subspaces}
We study multivariate integration 
\begin{gather}\label{Int}
		\mathrm{Int}_d f = \int_{[0,1]^d} f(\bsx) \rd \bsx
\end{gather}
for functions from subspaces of some Hilbert space of periodic functions 
\begin{align*}
 		F_{d}(r_{\alpha,\bm{\beta}}) = \left\{ f:[0,1]^d \rightarrow \mathbb{C} 
 				\sep f\in L_2 \, \text{ with } \, \norm{f}_{d}^2 = \sum_{\bsh \in \Z^d} \abs{\widehat{f}(\bsh)}^2 r_{\alpha,\bm{\beta}}(\bsh) < \infty \right\}.
\end{align*}
Hence, functions in $f\in F_{d}(r_{\alpha,\bm{\beta}})$ can be represented 
in terms of an absolutely convergent Fourier expansion and their Fourier coefficients
\begin{align*}
		\widehat{f}(\bsh) = \distr{f}{\exp(\twopii \bsh \cdot \cdot)}_{L_2} 
		= \int_{[0,1]^d} f(\bsx) \, \exp(-\twopii \bsh \cdot \bsx) \rd\bsx, 
\end{align*}
$\bsh = (h_1,\ldots,h_d)\in\Z^d$, decay 
faster than
$r_{\alpha,\bm{\beta}}(\bsh)^{-1/2}$.
Here $r_{\alpha,\bm{\beta}} \colon \Z^d \nach (0,\infty)$
is a $d$-fold tensor product involving some function $R\colon [1,\infty) \nach (0,\infty)$ and 
a tuple $\bm{\beta} = (\beta_0,\beta_1)$ of positive parameters
such that
\begin{gather*}
		r_{\alpha,\bm{\beta}}(\bsh) 
		= \prod_{\ell=1}^d \left( \delta_{0,h_\ell} \, \beta_0^{-1} + (1-\delta_{0,h_\ell}) \, \beta_1^{-1} \, R(\abs{h_\ell})^{2\alpha} \right), 
		\quad \bsh\in\Z^d.
\end{gather*}
Therein $\delta$ denotes the Kronecker delta (i.e., $\delta_{i,j}$ equals one if $i=j$ and zero otherwise) and the parameter $\alpha \geq 0$ describes the smoothness.
Throughout the whole paper we assume that 
\begin{equation*}
	\frac{1}{c_R} \,R(m) \leq \frac{R(n m)}{n} \leq R(m)
	\qquad \text{for all} \quad n \geq 1, \quad m\in\N, \quad \text{and some} \quad c_R\geq 1.
\end{equation*}
Moreover, we assume that
$(R(m)^{-1})_{m\in\N} \in \ell_{2\alpha}$, \ie
\begin{gather}\label{summableR}
		\NR{\alpha}=\sum_{m=1}^\infty \frac{1}{R(m)^{2\alpha}} < \infty
\end{gather}
(Note that the latter conditions particularly imply that $R(m)\sim m$ and $\alpha>1/2$).

For a detailed discussion of $F_{d}(r_{\alpha,\bm{\beta}})$ 
we refer to Novak and Wo{\'z}niakowski~\cite[Appendix~A.1]{NW08}
but we want to stress the point that some well-known spaces are covered by this definition. 

\begin{example}\label{ex:spaces}
\noindent
\begin{itemize}
\item[(i)] For $\beta_0=\beta_1=1$, $\alpha > 1/2$ and $R(m)=m$, $m\in\N$, we obtain the \emph{classical Korobov space},
where $r_{\alpha,\bm{\beta}}(\bsh) = \prod_{\ell=1}^d \max\{1, |h_\ell|\}^{2\alpha}$.
\item[(ii)] If we change our definition of $R$ to $R(m)=2\pi m$, $m\in\N$, and assume that $\alpha\in\N$ 
then, for any positive $\beta_0$ and $\beta_1$, we have a norm which resembles that of the \emph{unanchored Sobolev space} restricted to periodic functions where the norm for $d=1$ can also be written as
\begin{gather*}
		\norm{f}_1^2 = \beta_0^{-1} \abs{\int_0^1 f(x) \rd x}^2 + \beta_1^{-1} \int_0^1 \abs{f^{(\alpha)}(x)}^2 \rd x.
\end{gather*}
\item[(iii)] Also the periodic \emph{Sobolev space of dominating mixed smoothness} 
$S_2^\alpha W$ studied, e.g., in Ullrich~\cite{Tino}, is covered.
To this end, let $\beta_0=\beta_1=1$, $\alpha > 1/2$  
and $R(m)=(1+m^2)^{1/2}$ for $m\in\N$. Then $c_R=\sqrt{2}$ 
and
\begin{gather*}
		\norm{f}_d^2 = \sum_{\bsh \in \Z^d} \left( \abs{\widehat{f}(\bsh)}  \prod_{\ell=1}^d (1+\abs{h_\ell}^2)^{\alpha/2} \right)^2. \qedhere
\end{gather*}
\end{itemize}
\end{example}

Due to $R(m)\sim m$ it is known that if $\alpha>1/2$ then 
$F_d(r_{\alpha,\bm{\beta}})$ is a $d$-fold tensor product 
of some univariate reproducing kernel Hilbert space (RKHS)
equipped with the inner product
\begin{align*}
		\distr{f}{g} 
		= \sum_{\bsh \in \Z^d} r_{\alpha,\bm{\beta}}(\bsh) \, \widehat{f}(\bsh)\, \overline{\widehat{g}(\bsh)}.
\end{align*}
Thus, $F_d(r_{\alpha,\bm{\beta}})$ itself is also a RKHS, where the respective $d$-variate kernel is given by
\begin{gather}\label{def_kernel}
		K_d(\bsx,\bsy) 
		= \sum_{\bsh \in \Z^d} r_{\alpha,\bm{\beta}}(\bsh)^{-1} \exp \left(\twopii \bsh \cdot (\bsx-\bsy) \right).
\end{gather}
A comprehensive discussion of RKHSs can be found in Aronszajn~\cite{A50}.
For the latest state of the art in integration theory related to RKHSs we refer the reader
to the textbook of Dick and Pillichshammer \cite{DP10}, as well as to the survey article of
Dick, Kuo, and Sloan \cite{DKS13} and the references therein. 
A detailed introduction to special integration methods (such as lattice rules discussed below) can also be found in the monographs of Sloan and Joe~\cite{SJ94} and Novak and Wo{\'z}niakowski~\cite{NW10}, as well as in the review \cite{Nuy2014}. 

In what follows we focus on the integration problem restricted to 
subsets of \emph{$I_d$-permutation-invariant functions} $f\in F_d(r_{\alpha,\bm{\beta}})$ 
for some coordinate sets $I_d \subseteq \{1,\ldots,d\}$; see \cite{W12,W14}.
That is, we impose the additional condition that $f$ is invariant under all permutations of the variables with indices in $I_d$:
\begin{align*}
 		f(\bsx) &= f(P(\bsx)) 
 		\qquad \text{for all } \quad \bsx\in[0,1]^d \quad \text{ and each } \quad P \in \S_d,
\end{align*}
where
\begin{align*}
		\S_d 
		&= \S_{\{1,\ldots,d\}}(I_d) 
		= \left\{ P\colon\{1,\ldots,d\}\nach\{1,\ldots,d\} \sep P \text{ bijection such that } P\big|_{\{1,\ldots,d\}\setminus I_d} = \id \right\}.  
\end{align*}
(Note that this set always contains at least the identity permutation.)
These subspaces will be denoted by $\SI_{I_d}(F_d(r_{\alpha,\bm{\beta}}))$.
For the extremal case of \emph{fully permuta\-tion-invariant functions} we use the 
shorthand
$\SI(F_d(r_{\alpha,\bm{\beta}}))$.
It is known that if $I_d = \{i_1, i_2, \ldots, i_{\# I_d}\}$ 
then the subset of symmetrized and scaled basis functions of $F_d(r_{\alpha,\bm{\beta}})$ given by
\begin{gather*}
 		\phi_{\bsk}(\bsx) 
 		= \sqrt{\frac{r^{-1}_{\alpha,\bm{\beta}}(\bsk)}{\#\S_d \; \M_d(\bsk)!}} \sum_{P \in \S_{d}} \exp(\twopii P(\bsk)\cdot \bsx)
\end{gather*}
with 
\begin{gather}\label{def_nabla}
 		\bsk \in \nabla_d 
 		= \nabla_{\{1,\ldots,d\}}(I_d)
 		= \left\{\bsk \in \Z^d \sep k_{i_1} \leq k_{i_2} \leq \cdots \leq k_{i_{\# I_d}} \right\},
\end{gather}
builds an orthonormal basis of $\SI_{I_d}(F_d(r_{\alpha,\bm{\beta}}))$; see \cite{W12} for details.
Here
\begin{align*}
  \M_d(\bsk)!
  =
  \M_{\{1,\ldots,d\}}(\bsk, I_d)!
  &=
  \#\{ P \in \S_{d} \sep P(\bsk) = \bsk \}
\end{align*}
accounts for the repetitions of indices in the multi-index $\bsk$, giving rise to repetitive permutations. 
It immediately follows that for every function $G\colon \Z^d\nach \C$ it holds
\begin{equation}\label{eq:transform}
	\sum_{\bsh \in \Z^d} G(\bsh)
	= 
	\sum_{\bsk \in \nabla_d} \frac{1}{\M_d(\bsk)!} \sum_{P \in \S_d} G(P(\bsk))
	.
\end{equation}
Since $\SI_{I_d}(F_d(r_{\alpha,\bm{\beta}}))$ is equipped with the same norm as the entire space 
$F_d(r_{\alpha,\bm{\beta}})$ it is again a RKHS.
Moreover, it can be easily checked that its reproducing kernel is given by
\begin{align}
		K_{d,I_d}(\bsx,\bsy) 
		&= \frac{1}{\left( \# \S_d \right)^2} \sum_{P,P' \in \S_d} K_d(P(\bsx),P'(\bsy)) 
		 = \frac{1}{\#\S_d} \sum_{P \in \S_d} K_d(P(\bsx),\bsy) 
		\nonumber\\
        &= \sum_{\bsk \in \nabla _d} \frac{r^{-1}_{\alpha,\bm{\beta}}(\bsk)}{\M_d(\bsk)!} \frac{1}{\#\S_d} \sum_{P,P' \in \S_d} \exp(\twopii \bsk \cdot (P(\bsx) - P'(\bsy))) \nonumber \\
		&= \sum_{\bsh \in \Z^d} \frac{r^{-1}_{\alpha,\bm{\beta}}(\bsh)}{\#\S_d} \sum_{P \in \S_d} \exp(\twopii \bsh \cdot (P(\bsx)-\bsy)) \label{sym_kernel_2} 
\end{align}
for $\bsx,\bsy\in[0,1]^d$. Finally, it is known that (using a suitable rearrangement of the coordinates) the space
$\SI_{I_d}(F_d(r_{\alpha,\bm{\beta}}))$ can be seen as the tensor product of
the fully permutation-invariant subset
of the $\# I_d$-variate space with the entire $(d-\# I_d)$-variate space,
\ie
\begin{gather*}
		\SI_{I_d}(F_d(r_{\alpha,\bm{\beta}})) 
		= \SI(F_{\# I_d}(r_{\alpha,\bm{\beta}})) \otimes F_{d-\# I_d}(r_{\alpha,\bm{\beta}}).
\end{gather*}
Hence, also the reproducing kernel factorizes to
\begin{gather}\label{tensor_kernel}
		K_{d,I_d} = K_{\# I_d, \{1,\ldots,\# I_d\}} \otimes K_{d-\#I_d}.
\end{gather}

\begin{remark}
Some comments are in order.
\begin{itemize}
\item[(i)] Note that our theory can be extended easily to spaces which yield permutation-invariance 
with respect to at least two disjoint subsets of coordinates $I_d$ and $J_d$. 
Similar spaces play some role for approximation problems from computational practice, e.g., related to the electronic Schr\"{o}dinger equation; see \cite{W12}.
\item[(ii)] We do not consider anisotropic spaces $F_d(r_{\alpha,\bm{\beta}})$ where the parameters $\beta_1$
in $r_{\alpha,\bm{\beta}}$ are allowed to depend on the index of the respective variable.
Although this approach is reasonable to model the influence of different variables $x_j$ on $f(\bsx)$,
when $j \in I_d$
the effect would be averaged out by the application of the permutations $P\in\S_d$ such that
finally all variables in $I_d$  
would be equally important.
The same result can be reached by taking appropriate constant values of $\beta_1$.
For $j \notin I_d$ the standard results apply and so we do not study this here; see, e.g., \cite{NW10} or Sloan and Wo{\'z}niakowski~\cite{SW98}.
\item[(iii)] In this paper we mainly concentrate on spaces with weight parameters $\beta_1$ that are independent of the dimension $d$. 
For tractability it turns out that this case is sufficient, provided that the number of permutation-invariance conditions (i.e., the cardinality of the sets $I_d$) is large enough. Occasionally we briefly describe how to proceed if this major assumption is violated.
\hfill$\qedhere$
\end{itemize}
\end{remark}

\subsection{Algorithms, worst case errors and notions of tractability}
We like to approximate the integral \link{Int} by some weighted cubature rule
\begin{gather}\label{QMC}
		Q_{d,n}(f) 
		= Q_{d,n}\!\left(f; \bm{t}^{(0)},\ldots,\bm{t}^{(n-1)}, w_0,\ldots,w_{n-1}\right) 
		= \frac{1}{n} \sum_{j=0}^{n-1} w_j \, f\!\left(\bm{t}^{(j)}\right),
\end{gather}
$d,n\in\N$, that samples $f$ at the points $\bm{t}^{(j)}\in[0,1]^d$, $j=0,\ldots,n-1$, where
the weights $w_j$ are well-chosen real numbers.
If $w_0=\cdots=w_{n-1}=1$ and all $\bm{t}^{(j)}$ are chosen deterministically then $Q_{d,n}$ is the classical \emph{quasi-Monte Carlo} (QMC) rule which we will denote by $\mathrm{QMC}_{d,n} = \mathrm{QMC}_{d,n}(\,\cdot\,; \bm{t}^{(0)},\ldots,\bm{t}^{(n-1)})$.
This construction is inspired by the standard \emph{Monte Carlo} algorithm $\mathrm{MC}_{d,n}$ that formally equals $\mathrm{QMC}_{d,n}$ with the difference that here the sample points $\bm{t}^{(j)}$, $j=0,\ldots,n-1$, are independent and identically uniformly distributed in $[0,1]^d$.

Provided that $K$ is the reproducing kernel of some
RKHS~$H_d$ of functions on $[0,1]^d$ the squared worst case error of $Q_{d,n}$ is then given by, see, e.g., Hickernell and Wo{\'z}niakowski~\cite{HW00},
\begin{align}
		e^{\mathrm{wor}}(Q_{d,n}; H_d)^2 
		&= \left( \sup_{\substack{f\in H_d \\ \norm{f}_{d}\leq 1}} \abs{\mathrm{Int}_d f - Q_{d,n}(f)} \right)^2  \nonumber\\
		&= \int_{[0,1]^{2d}} K(\bsx,\bsy) \rd\bsx \rd\bsy - \frac{2}{n} \sum_{j=0}^{n-1} w_j \int_{[0,1]^d} K\!\left(\bsx,\bm{t}^{(j)}\right) \rd\bsx  \nonumber\\
		&\qquad \qquad+ \frac{1}{n^2} \sum_{j,\ell=0}^{n-1} w_j w_\ell K\!\left(\bm{t}^{(j)},\bm{t}^{(\ell)} \right).
		\label{eq:e2-K}
\end{align}
In what follows, we want to bound the $n$th \emph{minimal worst case error} 
\begin{gather}\label{def_minerror}
		e(n,d; H_d) = \inf_{A_{n,d}} e^{\mathrm{wor}}(A_{n,d}; H_d)
\end{gather}
for integration on $H_d$.
Here the infimum is taken with respect to some class of algorithms $A_{n,d}$ which use at most 
$n$ samples of the input function. 

\begin{remark}
We stress that due to results of Smolyak and Bakhvalov we can restrict ourselves to linear, non-adaptive
cubature rules $Q_{d,n}$ of the form \link{QMC}, without loss of generality. For details and
further references see, e.g., Sloan and Wo{\'z}niakowski~\cite[Remark~1]{SW97} or \cite[Section~4.2.2]{NW08}.
\end{remark}

In this context, we briefly recall the concepts of tractability that will be used later on. For this purpose we rely on the notions described in~\cite{NW08}. Let $n = n(\varepsilon, d)$ denote the \emph{information complexity} with respect to the \emph{normalized error criterion}. That is, the minimal number of function values necessary
to reduce the \emph{initial error} $e(0,d; H_d)$ by a factor of $\varepsilon \in (0,1)$, in the $d$-variate case. 
Then a problem is said to be \emph{polynomially tractable} if $n(\varepsilon ,d)$ is upper bounded by some polynomial in $\varepsilon^{-1}$
and $d$, \ie if there exist constants $C,p>0$, and $q \geq 0$  such that for all $d\in\N$ and every $\varepsilon\in(0,1)$
\begin{align}
\label{trac_defn}
	 n(\varepsilon ,d) \le C\, d^q \, \varepsilon^{-p}.
\end{align}
If this bound is independent of $d$, \ie if we can take $q=0$, then the problem is said to be \emph{strongly polynomially tractable}. In contrast, problems are called \emph{polynomially intractable} if \eqref{trac_defn} does not hold for any such choice of $C,p$, and $q$.
Finally, a problem is said to be \emph{weakly tractable} if its information complexity does not grow exponentially with $\varepsilon^{-1}$ and $d$, \ie if
\begin{equation*}
	\lim_{\varepsilon^{-1}+d\nach\infty} \frac{\ln n(\varepsilon,d)}{\varepsilon^{-1}+d}=0.
\end{equation*}


\section{Upper bounds and tractability}\label{sect:avgbounds}
Here we derive conditions on the problem parameters $I_d$ and $r_{\alpha,\bm{\beta}}$ that are sufficient to guarantee (strong) polynomial tractability of the integration problem under consideration.
To this end, we recall an averaging technique that allows to establish upper bounds on the $n$th minimal worst case error \link{def_minerror}. 
Arguments of this type were initially presented in \cite[Lemma 8]{SW98} and further developed by Plaskota, Wasilkowski, and Zhao~\cite{PWZ09}.
For generalizations of the method the interested reader is referred to~\cite[Section 10.7]{NW10}. 


\subsection{An averaging technique}
Given a reproducing kernel $K$ let us define the quantities
\begin{align*}
 		M_{1,d} &= M_{1,d}(K) = \left( \int_{[0,1]^d} \sqrt{K(\bsx,\bsx)} \rd\bsx \right)^2, \\
 		M_{2,d} &= M_{2,d}(K) = \int_{[0,1]^d} K(\bsx,\bsx) \rd\bsx, 
\end{align*}
and
\begin{gather*}
		S_d = S_d(K) = \int_{[0,1]^d}\int_{[0,1]^d} K(\bsx,\bsy) \rd\bsx \rd\bsy.
\end{gather*}
Then $S_d$ coincides with the square of the initial error of numerical integration over $H_d=H(K)$ with respect to the worst case setting.
Furthermore, it can be checked that
\begin{equation*}
	S_d \leq M_{1,d} \leq M_{2,d}.
\end{equation*}
Therefore the integration problem is well-defined for $H_d$ if at least $M_{2,d}(K)$ is finite and it is normalized if $S_d=1$.

The following 
result can be found in \cite[Theorem~1]{PWZ09}.
\begin{proposition}\label{prop:avg_bounds}
		Let $n\in\N$ and assume $M_{1,d}<\infty$ for all $d\in\N$. Then
		\begin{gather*}
				e(n,d;H_d) \leq \sqrt{M_{1,d}-S_d} \; n^{-1/2} = \sqrt{\frac{M_{1,d}}{S_d}-1} \; n^{-1/2} \; e(0,d;H_d).
		\end{gather*}
		and there exist points $\bm{t}^{(0)},\ldots,\bm{t}^{(n-1)}\in[0,1]^d$ 
		such that the cubature rule $Q_{d,n}$ with $w_i=\sqrt{M_{1,d}/K(\bm{t}^{(i)},\bm{t}^{(i)})}$
		achieves this bound.
		Moreover, if $M_{2,d}<\infty$ for $d\in\N$, then there are points such that $\mathrm{QMC}_{d,n}$ (i.e., $w_i \equiv 1$) satisfies
		\begin{gather*}
				e^{\mathrm{wor}}(\mathrm{QMC}_{d,n};H_d) \leq \sqrt{M_{2,d}-S_d} \; n^{-1/2} = \sqrt{\frac{M_{2,d}}{S_d}-1} \; n^{-1/2} \; e(0,d;H_d).
		\end{gather*}
\end{proposition}

\begin{remark}\label{rem:semiconstr}
Although these bounds are non-constructive it is known that slightly larger bounds
can be achieved with high probability by any random set of points; see \cite[Remark 2]{PWZ09}.
\end{remark}

We want to apply \RefProp{prop:avg_bounds} for the spaces $H_d = \SI_{I_d}(F_d(r_{\alpha,\bm{\beta}}))$ as defined in \autoref{sect:setting}.
In order to conclude (strong) polynomial tractability we simply need to bound $M_{2,d}/S_d$ from above by $C \, d^q$ for some $C,q \geq 0$ and all $d\in\N$ 
(with $q=0$ for strong polynomial tractability).
In the following lemma we calculate the quantities of interest. We postpone its proof to the appendix in \autoref{sect:appendix}.

\begin{lemma}\label{lem:SdMd}
	For $d\in\N$ and every $I_d\subseteq\{1,\ldots,d\}$ it holds
 	\begin{equation}\label{eq:M2Id}
		S_d(K_{d,I_d}) = \beta_0^d
		\qquad \text{and} \qquad
		M_{2,d}(K_{d,I_d})
		= \sum_{\bsk \in \nabla_d} r^{-1}_{\alpha,\bm{\beta}}(\bsk).
	\end{equation}
	If $\#I_d<2$, i.e., $K_{d,I_d}=K_d$, we moreover have, with $\NR{\alpha}$ defined in\RefEq{summableR},
	\begin{equation}\label{eq:M1}
		M_{1,d}(K_d)
		= M_{2,d}(K_d) 
		= \beta_0^d  \left(1 + \frac{2\beta_1 \NR{\alpha}}{\beta_0} \right)^d.
	\end{equation}
\end{lemma}

\begin{remark}
We stress the point that, since $\nabla_d \subsetneq \Z^d$ whenever $\#I_d \geq 2$, the term $M_{2,d}(K_{d,I_d})$ given in \link{eq:M2Id} might be dramatically smaller than the respective quantity \link{eq:M1} for the full space.
\end{remark}

In order to derive a suitable upper bound for $M_{2,d}(K_{d,I_d})/S_d$ it suffices to 
consider the fully permutation-invariant part. That is, we assume $K=K_{d,\{1,\ldots,d\}}$ in what follows.
Denoting the number of non-zero components $h_j$ of $\bsh\in\Z^d$ by $|\bsh|_0$
we can estimate the sum in \link{eq:M2Id} as follows:
\begin{align*}
	M_{2,d}(K_{d,\{1,\ldots,d\}}) 
    &= \frac{1}{\#\S_d} \sum_{\bsh \in \Z^d} \M_d(\bsh)! \, r^{-1}_{\alpha,\bm{\beta}}(\bsh) \\
    &\leq \frac{1}{\#\S_d} \sum_{\bsk \in \N_0^d} 2^{|\bsk|_0} \M_d(\bsk)! \, r^{-1}_{\alpha,\bm{\beta}}(\bsk)
    = \sum_{\substack{\bsk\in\N_0^d\\0\leq k_{1}\leq\cdots\leq k_{d}}} 2^{|\bsk|_0} \, r^{-1}_{\alpha,\bm{\beta}}(\bsk).
\end{align*}
The latter sum can be bounded with the help of another, rather technical lemma which is based on \cite[Lemma~4]{W12}. For the convenience of the reader a detailed proof can be found in the appendix (\RefSec{sect:appendix}).
\begin{lemma}\label{lemmaBound}
				Assume $(\lambda_m)_{m\in\N_0}$ to be a sequence of 
				non-negative real numbers with $\lambda_0 > 0$ and $\lambda_0 \geq \lambda_m \ge 0$ for all $m\in\N_0$.
				Moreover, set $\lambda_{s,\bm{k}}=\prod_{\ell=1}^s \lambda_{k_\ell}$ for $\bm{k}\in\N_0^s$ and $s\in\N$.
				Then we have for all $V\in\N_0$ and every $d\in\N$
				\begin{gather}\label{estimate_VNEW}
								\sum_{\substack{\bm{k}\in\N_0^d\\0\leq k_1\leq\cdots\leq k_d}} \lambda_{d,\bm{k}}
								\leq \lambda_0^d \, d^V \left( 1 + V + \sum_{L=1}^d \lambda_0^{-L} 
										 \sum_{ \substack{ \bm{j}\in\N^L\\V+1\leq j_1\leq\cdots\leq j_L } } \lambda_{L,\bm{j}} \right)
				\end{gather}
				with equality at least for $V=0$.
\end{lemma}

\noindent Setting $(\lambda_m)_{m\in\N_0}$ to
\begin{gather}\label{mu}
		\lambda_0 = \beta_0>0 \quad\text{and}\quad \lambda_m = 2 \, \beta_1 \, R(m)^{-2\alpha}, \quad m\in\N,
\end{gather}
we observe that $\lambda_{d,\bm{k}} = 2^{\abs{\bm{k}}_0} \, r^{-1}_{\alpha,\bm{\beta}}(\bsk)$, $\bm{k}\in\N_0^d$. 
Thus, we can apply \RefLem{lemmaBound} if 
\begin{gather}\label{assump}
		\frac{2 \, \beta_1}{\beta_0 \, R(m)^{2\alpha}} \leq 1
		\quad \text{for all} \quad m\in\N.
\end{gather} 
In Equation\RefEq{eq:repres} from \RefProp{prop:constant} we will see that a condition like \link{assump} is indeed necessary in order to avoid an exponential dependence of the term $M_{2,d}(K_{d,\{1,\ldots,d\}}) /S_d$ on the dimension~$d$.
From~\link{summableR} we particularly conclude that there
exists some $V^*=V^*(R,\alpha,\bm{\beta})\in\N_0$ such that
\begin{gather}\label{boundN}
		\NN^* = \NN^*(V^*)= \sum_{m=V^*+1}^\infty \frac{2 \, \beta_1}{\beta_0 \, R(m)^{2\alpha}} < 1.
\end{gather}
Using \RefLem{lemmaBound} for this $V^*$ and $\lambda$ given by \link{mu}
we obtain
\begin{align*}
		M_{2,d}(K_{d,\{1,\ldots,d\}}) 
		&\leq \beta_0^d \, d^{V^*} \left( 1+ V^* + \sum_{L=1}^d \sum_{ \substack{ \bm{j}\in\N_0^L \\V^*+1\leq j_1\leq\cdots\leq j_L } }  
		  \prod_{\ell=1}^L \frac{2 \, \beta_1}{\beta_0 \, R(j_\ell)^{2\alpha}} \right) \\
		&\leq \beta_0^d \, d^{V^*} \left( 1+ V^* + \sum_{L=1}^\infty (\NN^*)^L  \right) \\
		&= S_d \, d^{V^*} \left( V^* + \frac{1}{1-\NN^*} \right).
\end{align*}
In summary we obtain the bound
\begin{align*}
	\frac{M_{2,d}(K_{d,I_d})}{S_d}
	&= \frac{M_{2,d-\# I_d}(K_{d-\# I_d})}{S_{d-\# I_d}} \; \frac{ M_{2,\# I_d}(K_{\# I_d, \{1,\ldots,\#I_d\}})}{S_{\# I_d}} \\
	&\leq \left(1 + \frac{2\beta_1 \NR{\alpha}}{\beta_0} \right)^{d-\# I_d} \; (\# I_d)^{V^*} \left( V^* + \frac{1}{1-\NN^*} \right)
\end{align*}
which, in view of \RefProp{prop:avg_bounds}, implies the following theorem that ensures the existence of good QMC algorithms for the approximation of the integrals \link{Int}.

\begin{theorem}\label{thm:tractability}
		For $d\geq 2$ let $I_d\subseteq\{1,\ldots,d\}$ with $\# I_d \geq 2$ and assume \link{assump} to be true.
		We consider the integration problem on the $I_d$-permutation-invariant subspaces $\SI_{I_d}(F_d(r_{\alpha,\bm{\beta}}))$.
		Then
		\begin{itemize}
			\item[$\bullet$] for all $n$ and $d\in\N$ the $n$th minimal worst case error is bounded by
		\begin{align}
			e(n,d; \SI_{I_d}(F_d(r_{\alpha,\bm{\beta}}))) 
			&
			\leq e(0,d; \SI_{I_d}(F_d(r_{\alpha,\bm{\beta}}))) \, \sqrt{V^* + \frac{1}{1-\NN^*}}  \nonumber \\
				&\qquad\qquad \times \left(1 + \frac{2 \beta_1 \NR{\alpha}}{\beta_0} \right)^{(d-\#I_d)/2} (\#I_d)^{V^*/2} \frac{1}{\sqrt{n}}, \label{errorBound}
		\end{align}
		where the absolute constants $V^*$ and $\NN^*$ are given by \link{boundN}.
			\item[$\bullet$] there exists a QMC rule which achieves this bound.
		\end{itemize}
		Consequently, we have the following tractability statements:
		\begin{itemize}
			\item[$\bullet$] If $(d-\#I_d) \in \0(\ln d)$ then the integration problem is polynomially tractable (with respect to the worst case setting and the normalized error criterion).
			\item[$\bullet$] If $(d-\#I_d) \in \0(1)$ and \link{boundN} holds for $V^*=0$ then we obtain strong polynomial tractability.
		\end{itemize}		
\end{theorem}

\subsection{Discussion}
Let us illustrate the obtained results with some examples.
We first consider the case where $I_d=\{1,\ldots,d\}$ and $\beta_0=1$, \ie fully permutation-invariant subspaces
where the integration problem is well-scaled.
In this case the bound~\link{errorBound} simplifies to
\begin{gather*}
		e(n,d) 
				\leq \sqrt{V^* + \frac{1}{1-\NN^*}} \; d^{V^*/2} \; \frac{1}{\sqrt{n}}.
\end{gather*}

Then for the classical unweighted Korobov space ($\beta_0=\beta_1=1$ and $R(m)=m$, see \RefEx{ex:spaces}(i)) our assumption~\link{assump} is not
fulfilled.
We can overcome this problem by changing the parameter $\beta_1$ to $1/2$.
In this case $\NN^*$ equals the generalized zeta function $\zeta(2\alpha,V^*+1)$ 
which can only be smaller than one for $V^*>0$, depending on $\alpha$.
Hence, we can show polynomial tractability, but not strong polynomial tractability for the Korobov space.

For the periodic unanchored Sobolev space from \RefEx{ex:spaces}(ii) with
$\beta_0=\beta_1=1$ and $R(m)=2\pi m$ our assumption~\link{assump} is always fulfilled for $\alpha > 1/2$ and we can prove strong polynomial tractability if $\alpha$ is sufficiently large.
(We here consider arbitrary $\alpha$ as this is then a modified Korobov space with appropriately chosen $\beta_1$.)
Indeed, $\NN^*(0)<1$ if $\alpha \geq \alpha^*\approx 0.61769976$.
Unfortunately, the constant $(1-\NN^*)^{-1/2}$ 
will be extremely large for smoothness parameters 
$\alpha$ close to $\alpha^*$.
On the other hand, already $\alpha=1$ yields
\begin{gather*}
		\sqrt{\frac{1}{1-\NN^*}} = \sqrt{12/11} \approx 1.044465936.
\end{gather*}

For the periodic Sobolev space of dominating mixed smoothness discussed in \RefEx{ex:spaces}(iii), with 
$\beta_0=\beta_1=1$ and $R(m) = (1+m^2)^{1/2}$, it follows immediately that our assumption~\link{assump} is fulfilled for $\alpha \ge 1$.
For $\alpha=1$ and $V^*=0$ we have $\eta^*(0) \approx 2.15335$, but for $\alpha \ge \alpha^* \approx 1.521196$ we find $\eta^*(0) < 1$ and thus strong polynomial tractability from there on.
For $\alpha=2$ we thus have strong polynomial tractability with $\eta^*(0) \approx 0.613674$ and a very acceptable constant of $(1-\NN^*)^{-1/2} \approx 1.60888$.
Although we only replace $\max\{1,|k|\}$ by $(1+k^2)^{1/2}$ for $k \in \Z$ when moving from the Korobov norm to the norm of the Sobolev space of dominating mixed smoothness, these observations particularly show that exchanging equivalent norms can cause a big difference in tractability.

We contrast these results with results known for the full space, i.e., for $I_d=\emptyset$.
In the literature the following assertions for the full space with
\begin{equation*}
	\alpha>1/2, 
	\quad 
	\beta_0=1, 
	\quad 
	0 < \beta_1=\beta_1(d) \leq C < \infty, 
	\quad \text{and} \quad 
	R(m)=2\pi m \, \text{ for } \, m\in\N,
\end{equation*}
can be found; see \cite[Theorems~16.5 and~16.16]{NW10}.
\begin{proposition}\label{prop:fullspace}
	Integration on $F_d(r_{\alpha,\bm{\beta}})$ is
	\begin{itemize}
		\item[$\bullet$] weakly tractable if and only if 
			\begin{equation*}
				\lim_{d\nach\infty} \beta_1(d)=0.
			\end{equation*}
		\item[$\bullet$] polynomially tractable if and only if 
			\begin{equation*}
				\beta_1(d) \leq C\, \frac{\ln(d+1)}{d}
				\qquad \text{for some} \quad C<\infty \quad \text{and all} \quad d\in\N.
			\end{equation*}
		\item[$\bullet$] strongly polynomially tractable if and only if
			\begin{equation*}
				\beta_1(d) \leq C\, \frac{1}{d}
				\qquad \text{for some} \quad C<\infty \quad \text{and all} \quad d\in\N.
			\end{equation*}
	\end{itemize}
\end{proposition}
\begin{proof}
The authors of \cite{NW10} deal with coordinate dependent bounded product weights $\gamma_{d,j}$. Setting $\gamma_{d,j} = \beta_1(d)$ for all $j=1,\ldots,d$ and every $d\in\N$ proves the claim.
\end{proof}

Thus if $\alpha$ is large enough, then we have strong polynomial tractability for the fully permutation-invariant problem, whereas the integration problem on the full space is not even weakly tractable.

\begin{remark}\label{rem:decaying_weights}
Let us stress the point that there is a trade-off between our growth conditions on the subsets $I_d$ and the decay conditions on the weight parameters $\beta_1$ which are typically imposed to achieve tractability. To give an example, we see that the factor
\begin{align*}
	\left(1 + \frac{2 \beta_1 \NR{\alpha}}{\beta_0} \right)^{(d-\#I_d)/2} 
	\leq \exp\!\left( \frac{\NR{\alpha}}{\beta_0} \; \beta_1 \; (d-\#I_d) \right)
\end{align*}
in \link{errorBound} is upper bounded polynomially in $d$ if
\begin{equation*}
	\beta_1
	=\beta_1(d) 
	\leq C \,\frac{\ln (d+1)}{\max\{d-\#I_d,1\}}
\end{equation*}
for some $C>0$ and all $d$. If $\# I_d$ is uniformly bounded, then this condition coincides with the well-known assumption stated in \RefProp{prop:fullspace}. Moreover, in this case \link{assump} is always fulfilled (at least for $d\geq d_0$).
In contrast, allowing a growth of $I_d$ with the dimension leads us to weaker restrictions on $\beta_1$ such that finally constant $\beta_1$ is sufficient for (strong) polynomial tractability provided that \link{assump} is fulfilled.
\end{remark}


\section{Rank-1 lattice rules}
\label{sect:lattice_rules}
This section contains our main results. Here we analyze unshifted and shifted rank-$1$ lattice rules for the approximation of the integral \link{Int} of $I_d$-permutation-invariant functions from the Korobov-type spaces $F_d(r_{\alpha,\bm{\beta}})$ defined in \autoref{sect:setting}.
First of all we give an exact error formula which holds for \emph{general} cubature rules $Q_{d,n}$ of the form \link{QMC}. The proof can be found in the appendix (\autoref{sect:appendix}).

\begin{lemma}\label{lem:e2}
For $d,n\in\N$ let $Q_{d,n}$ denote a general cubature rule given by \link{QMC}.
Then its worst case error on the $I_d$-permutation-invariant subspace of $F_d(r_{\alpha,\bm{\beta}})$ satisfies
\begin{align*}
		& e^{\mathrm{wor}}(Q_{d,n}; \SI_{I_d}(F_d(r_{\alpha,\bm{\beta}})))^2 \\
		&\qquad\qquad
		= r^{-1}_{\alpha,\bm{\beta}}(\bm 0) \left( 1 - \frac2n \sum_{j=0}^{n-1} w_j \right) 
		+ 	\sum_{\bsh \in \Z^d} r^{-1}_{\alpha,\bm{\beta}}(\bsh) \left( \frac1{n} \sum_{\ell=0}^{n-1} w_\ell\,  \exp(-\twopii \bsh\cdot \bst^{(\ell)})\right) \\
		&  \qquad\qquad\qquad \qquad \qquad \qquad \qquad \qquad \qquad \qquad \times
 \left( \frac{1}{\#\S_d} \sum_{P \in \S_d} \frac1{n} \sum_{j=0}^{n-1} w_j \,\exp(\twopii P(\bsh) \cdot \bst^{(j)}) \right). 
	\end{align*}
\end{lemma}

\begin{remark} 
Note that, as for the standard space, the first part of the squared worst case error only depends on $r^{-1}_{\alpha,\bm{\beta}}(\bm 0)$ and $w_j$. Thus, it cannot be reduced by permutation-invariance encoded by~$I_d$.
Moreover, for QMC rules this term simplifies as usual.
\end{remark}

Before we turn to (randomly) shifted rank-$1$ lattice rules let us consider unshifted rules first.

\subsection{Lower bounds for unshifted rules}
\label{subsect:unshifted_lattice rules}
Given natural numbers $n$ and $d$, an \emph{$n$-point rank-$1$ lattice rule} $Q_n(\bsz)$ is a QMC rule (i.e., it takes the form \link{QMC} with $w_0=\cdots=w_{n-1}=1$) which is fully determined by its generating vector $\bsz \in \Z_n^d=\{0,1,\ldots,n-1\}^d$. 
It uses points $\bst^{(j)}$ from an \emph{integration lattice} $L=L(\bsz)$ induced by $\bsz$:
\begin{align*}
  \bst^{(j)}
  &=
  \left\{ \frac{\bsz j}{n} \right\}
  =
  \frac{\bsz j}{n} \bmod{1}
  \qquad \text{for} \qquad 
  j = 0, 1, \ldots, n-1
  .
\end{align*}
This choice is reasonable since we have the following \emph{character property} over $\Z_n^d$ w.r.t.\ the trigonometric basis:
\begin{equation}\label{eq:char}
  \frac1n \sum_{j=0}^{n-1} \exp(\twopii (\bsh \cdot \bsz) j / n)
  =
  \begin{cases}
    1 & \text{if $\bsh \cdot \bsz \equiv 0 \pmod{n}$}, \\
    0, & \text{otherwise}.
  \end{cases}
\end{equation}
As usual, we collect those $\bsh\in\Z^d$ for which this sum is one in the set $L^\bot$, called the \emph{dual lattice}.

\begin{proposition}\label{prop:latticeerror}
For $d\in\N$ let $Q_n(\bsz)$ denote an arbitrary (unshifted) rank-$1$ lattice rule as defined above. Then its worst case error on the $I_d$-permutation-invariant subspace of $F_d(r_{\alpha,\bm{\beta}})$ satisfies
	\begin{align*}
		e^{\mathrm{wor}}(Q_{n}(\bsz); \SI_{I_d}(F_d(r_{\alpha,\bm{\beta}})))^2 
		= \sum_{\bszero \ne \bsh \in L^\bot} \frac{r^{-1}_{\alpha,\bm{\beta}}(\bsh)}{\#\S_d} \sum_{P \in \S_d} \ind{P(\bsh) \in L^\bot}.
	\end{align*}
\end{proposition}
\begin{proof}
The proof directly follows from the definition of $Q_n(\bsz)$, formula~\link{eq:char}, and \RefLem{lem:e2}:
\begin{align*}
	&e^{\mathrm{wor}}(Q_{n}(\bsz); \SI_{I_d}(F_d(r_{\alpha,\bm{\beta}})))^2 
	= -r^{-1}_{\alpha,\bm{\beta}}(\bm 0) 
	+ 	\sum_{\bsh \in \Z^d} r^{-1}_{\alpha,\bm{\beta}}(\bsh) \; \ind{\bsh \in L^\bot} \left( \frac{1}{\#\S_d} \sum_{P \in \S_d} \ind{P(\bsh) \in L^\bot} \right).
	\qedhere
\end{align*}
\end{proof}
\begin{remark}
This expression also holds for general rank lattice rules.
\end{remark}

Denoting the $n$th minimal worst case error among all unshifted lattice rules by 
\begin{equation*}
	e_{\mathrm{lat}}(n,d; \SI_{I_d}(F_d(r_{\alpha,\bm{\beta}})))
	= \inf_{\bsz\in\Z_n^d} e^{\mathrm{wor}}(Q_{n}(\bsz); \SI_{I_d}(F_d(r_{\alpha,\bm{\beta}}))),
	\qquad d,n\in\N, 
\end{equation*}
we obtain the following negative result.

\begin{theorem}\label{thm:unshifted}
For every $d,n\in\N$ and all choices $I_d\subseteq\{1,\ldots,d\}$, it holds
\begin{align*}
	e_{\mathrm{lat}}(n,d; \SI_{I_d}(F_d(r_{\alpha,\bm{\beta}})))
	&\geq \left( \sum_{\bszero \ne \bsh \in \Z^d} r^{-1}_{\alpha,\bm{\beta}}(n\bsh) \right)^{1/2}\\
	&\geq e(0,d; \SI_{I_d}(F_d(r_{\alpha,\bm{\beta}}))) \! \left( \left[ 1 + \frac{2\beta_1 \NR{\alpha}}{\beta_0} \, \frac{1}{n^{2\alpha}} \right]^d - 1 \right)^{1/2}\!\!. 
\end{align*}
\end{theorem}
\begin{proof}
For any lattice rule $Q_n(\bsz)$ we always have that $n \Z^d \subseteq L^\bot$.
In view of \RefProp{prop:latticeerror} this establishes the lower bound
\begin{align*}
	&e^{\mathrm{wor}}(Q_{n}(\bsz); \SI_{I_d}(F_d(r_{\alpha,\bm{\beta}})))^2 
	\geq \sum_{\bszero \ne \bsh \in n\Z^d} \frac{r^{-1}_{\alpha,\bm{\beta}}(\bsh)}{\#\S_d} \sum_{P \in \S_d} \ind{P(\bsh) \in L^\bot}
	= \sum_{\bszero \ne \bsh \in \Z^d} r^{-1}_{\alpha,\bm{\beta}}(n\bsh). 
\end{align*}
The properties of $r_{\alpha,\bm{\beta}}$ and $R$ moreover yield
\begin{align*}
	\sum_{\bszero \ne \bsh \in \Z^d} r^{-1}_{\alpha,\bm{\beta}}(n\bsh)
	&=
	\prod_{\ell=1}^d \left[ \beta_0 + 2\beta_1 \sum_{m=1}^\infty R(n m)^{-2\alpha} \right] - \beta_0^d \\
	&= \beta_0^d \left( \left[ 1 + \frac{2\beta_1}{\beta_0} \sum_{m=1}^\infty R(n m)^{-2\alpha} \right]^d - 1 \right) \\
	&\geq \beta_0^d \left( \left[ 1 + \frac{2\beta_1 \NR{\alpha}}{\beta_0} \, \frac{1}{n^{2\alpha}} \right]^d - 1 \right).
\end{align*}
Since $\beta_0^d = S_d = e(0,d; \SI_{I_d}(F_d(r_{\alpha,\bm{\beta}})))^2$ taking the square root and passing to the infimum over all $\bsz\in\Z_n^d$ proves the claim.
\end{proof}

\begin{remark}\label{rem:Bernoulli}
Note that for fixed $n$ the term in the brackets grows exponentially in the dimension~$d$.
From Bernoulli's inequality it moreover follows that for all $d,n\in\N$
\begin{equation*}
	\left[ 1 + \frac{2\beta_1 \NR{\alpha}}{\beta_0} \, \frac{1}{n^{2\alpha}} \right]^d - 1
	\geq c \, d \, n^{-2\alpha},
\end{equation*}
where $c=2\beta_1 \NR{\alpha}/\beta_0$ is independent of $d$ and $n$.
Furthermore, this estimate is sharp (up to some absolute constant), provided that $n$ grows at least polynomially with $d$. To see this assume that $n$ satisfies $c\,d \, n^{-2\alpha} \leq c_1$ for some $0 < c_1 < 1$ and $c$ as before. 
Then $1+x \leq \exp(x)$ for all $x\geq 0$, and $\exp(y)\leq 1/(1-y)$ for all $y<1$, implies
\begin{align*}
	\left[ 1 + \frac{c}{n^{2\alpha}} \right]^d - 1
	\leq \exp\!\left( \frac{c\,d}{n^{2\alpha}} \right) - 1 
	\leq \frac{1}{1 - c\,d\, n^{-2\alpha}} - 1 
	&= \frac{c\,d\, n^{-2\alpha}}{1 - c\,d\, n^{-2\alpha}} 
	\leq \frac{c}{1-c_1} \, d \, n^{-2\alpha}
\end{align*}
which proves the claim.
\end{remark}

We derive the following tractability result which is in sharp contrast to \autoref{thm:tractability}.

\begin{corollary}
	Consider the integration problem on the $I_d$-permutation-invariant subspaces $\SI_{I_d}(F_d(r_{\alpha,\bm{\beta}}))$ in the worst case setting w.r.t. the normalized error criterion. Then
	\begin{itemize}
		\item[$\bullet$] the optimal rate of convergence which can be attained by unshifted lattice rules $Q_n(\bsz)$ is upper bounded by $\alpha$.
		\item[$\bullet$] independent of the problem parameters $I_d$ and $r_{\alpha,\bm{\beta}}$, the class of unshifted lattice rules $Q_n(\bsz)$ is too small to obtain strong polynomial tractability.
	\end{itemize}
\end{corollary}

\subsection{Existence of good shifted rank-$1$ lattice rules}
\label{subsect:shifted_lattice_rules}
In contrast to the negative result for unshifted lattice rules from the previous section, we will show here that there exist shifted lattice rules which satisfy the bound \link{errorBound} in \autoref{thm:tractability}.

Given $n$ and $d$, an $n$-point \emph{shifted rank-$1$ lattice rule} consists of an unshifted lattice rule $Q_n(\bsz)$, with generating vector $\bsz \in \Z_n^d$, whose points are shifted by some fixed $\bsDelta \in [0,1)^d$ modulo~1, i.e.,
\begin{equation*}
  \bst^{(j)}
  = \left\{ \frac{\bsz \,j}{n} + \bsDelta \right\}
  = \left( \frac{\bsz \, j}{n} + \bsDelta \right) \bmod{1}
  \qquad \text{for} \qquad j=0,1,\ldots,n-1.
\end{equation*}
In what follows such a cubature rule will be denoted by $Q_n(\bsz)+\bsDelta$.

To show that there exist good shifts $\bsDelta$ it is convenient to analyze the \emph{root mean squared worst case error}
\begin{equation*}
	E(Q_n(\bsz)) = \left( \int_{[0,1)^d} e^{\mathrm{wor}}(Q_{n}(\bsz)+\bsDelta; \SI_{I_d}(F_d(r_{\alpha,\bm{\beta}})))^2 \rd\bsDelta \right)^{1/2}
\end{equation*}
which is related to the \emph{shift-invariant kernel} (associated to $K_{d,I_d}$)
\begin{equation}\label{def_Kshinv}
  K_{d,I_d}^{\shinv}(\bsx, \bsy)
  =  \int_{[0,1)^d} K_{d,I_d}(\{\bsx+\bsDelta\},\{\bsy+\bsDelta\}) \rd\bsDelta,
  \qquad \bsx,\bsy\in[0,1]^d,
\end{equation}
as the next proposition shows.

\begin{proposition}\label{prop:Kshinv}
	Let $d\in\N$ and $I_d\subseteq\{1,\ldots,d\}$. Then the shift-invariant kernel can be written as
\begin{equation*}
	K_{d,I_d}^{\shinv}(\bsx, \bsy)
	= \sum_{\bsk \in \nabla_d} \frac{r_{\alpha,\bm{\beta}}^{-1}(\bsk)}{\#\S_d} \sum_{P \in \S_d} \exp\left(\twopii P(\bsk) \cdot (\bsx-\bsy)\right), \qquad \bsx,\bsy\in[0,1]^d.
\end{equation*}
Moreover, for every unshifted rank-$1$ lattice rule $Q_n(\bsz)$ we have
\begin{equation}\label{eq:mse}
	E(Q_n(\bsz))^2 
	= e^{\mathrm{wor}}(Q_{n}(\bsz); H_{d,I_d}^{\shinv})^2 
	= \sum_{\bm{0}\neq\bsk \in \nabla_d} \frac{r^{-1}_{\alpha,\bm{\beta}}(\bsk)}{\#\S_d} \sum_{P \in \S_d} \ind{P(\bsk) \in L^{\bot}},
\end{equation}
where $H_{d,I_d}^{\shinv}$ denotes the RKHS with kernel $K_{d,I_d}^{\shinv}$ and $L^{\bot}$ is the dual lattice induced by $\bsz\in\Z_n^d$.
\end{proposition}
\begin{proof}
\emph{Step 1.} From \link{def_Kshinv} and \link{sym_kernel_2} we derive that $K_{d,I_d}^{\shinv}(\bsx, \bsy)$ equals
\begin{align*}
  &\int_{[0,1)^d}
    \sum_{\bsh \in \Z^d} \frac{r_{\alpha,\bm{\beta}}^{-1}(\bsh)}{\#\S_d}
    \exp\left(-\twopii \bsh \cdot \{\bsy + \bsDelta\}\right)
    \sum_{P \in \S_d} \exp\left(\twopii P(\bsh) \cdot \{\bsx + \bsDelta\}\right)
  \rd\bsDelta
  \\
  &\qquad =
    \sum_{\bsh \in \Z^d} \frac{r_{\alpha,\bm{\beta}}^{-1}(\bsh)}{\#\S_d} \exp\left(-\twopii \bsh \cdot \bsy\right) 
 \sum_{P \in \S_d} \exp\left(\twopii P(\bsh) \cdot \bsx\right) \int_{[0,1)^d} \exp(\twopii [P(\bsh)-\bsh] \cdot \bsDelta) \rd\bsDelta, 
\end{align*}
where the latter integral is $1$ if $\bsh = P(\bsh)$ and $0$, otherwise. 
By definition, for $\bsh\in\Z^d$ there are exactly $\M_d(\bsh)!$ different permutations $P\in\S_d$ such that $\bsh = P(\bsh)$.
Consequently, using \link{eq:transform} we obtain that
\begin{align}
  \label{eq:KshinvZd}
  K_{d,I_d}^{\shinv}(\bsx, \bsy)
  &= \sum_{\bsh \in \Z^d} \frac{r_{\alpha,\bm{\beta}}^{-1}(\bsh)}{\#\S_d} \exp\left(\twopii \bsh \cdot (\bsx-\bsy)\right) \, \M_d(\bsh)! 
  \\
  \nonumber
  &= \sum_{\bsk \in \nabla_d} \frac{r_{\alpha,\bm{\beta}}^{-1}(\bsk)}{\#\S_d} \sum_{P \in S_d} \exp\left(\twopii P(\bsk) \cdot (\bsx-\bsy)\right)
\end{align}
for every $\bsx,\bsy\in[0,1]^d$.

\emph{Step 2.}
We use formula \link{eq:e2-K} for the worst case error in terms of the reproducing kernel, together with \link{eq:vanish_intK} from the proof of \RefLem{lem:e2} (see \autoref{sect:appendix}) and \link{def_Kshinv}, to obtain
\begin{align*}
	E(Q_n(\bsz))^2 
	&= -r^{-1}_{\alpha,\bm{\beta}}(\bszero) + \frac{1}{n^2} \sum_{j,\ell=0}^{n-1} \int_{[0,1)^d} K_{d,I_d}(\{\bm{s}^{(j)}+\bsDelta\},\{\bm{s}^{(\ell)}+\bsDelta\}) \rd\bsDelta \\
	&= -r^{-1}_{\alpha,\bm{\beta}}(\bszero) + \frac{1}{n^2} \sum_{j,\ell=0}^{n-1} K_{d,I_d}^{\shinv}(\bm{s}^{(j)}, \bm{s}^{(\ell)}) \\
	&= e^{\mathrm{wor}}(Q_{n}(\bsz); H_{d,I_d}^{\shinv})^2,
\end{align*}
where $\bm{s}^{(j)}=\{\bsz \, j / n\}$, $j=0,\ldots,n-1$, denote the nodes used by $Q_n(\bsz)$.
The rest of the claim now follows from the representation derived in Step 1.
\end{proof}

Subsequently, we deduce the existence of good shifts. At this point we restrict ourselves to lattice rules with a prime number of points as this simplifies proofs.
\begin{theorem}\label{thm:rmse}
	For $d\in\N$ let $I_d\subseteq\{1,\ldots,d\}$. Given a prime number $n\in\N$ let $Q_n(\bsz)$ denote an arbitrary (unshifted) rank-$1$ lattice rule for the integration problem on the $I_d$-permutation-invariant subspace of $F_d(r_{\alpha,\bm{\beta}})$. Then
	\begin{itemize}
	\item[$\bullet$] for some $\bsDelta^*=\bsDelta^*(\bsz)\in[0,1)^d$
	\begin{equation*}
		e^{\mathrm{wor}}(Q_{n}(\bsz)+\bsDelta^*; \SI_{I_d}(F_d(r_{\alpha,\bm{\beta}})))
		\leq E(Q_n(\bsz)) 
		\leq e^{\mathrm{wor}}(Q_{n}(\bsz); \SI_{I_d}(F_d(r_{\alpha,\bm{\beta}}))),
	\end{equation*}
	\ie there exists a shift such that $Q_{n}(\bsz)+\bsDelta^*$ performs better than $Q_n(\bsz)$.
	\item[$\bullet$] the root mean squared worst case error w.r.t.\ $\bsDelta\in[0,1)^d$ satisfies
	\begin{align*}
		E(Q_n(\bsz))		
		&\geq e(0,d; \SI_{I_d}(F_d(r_{\alpha,\bm{\beta}}))) \left( \left[ 1 + \frac{2\beta_1 \NR{\alpha}}{\beta_0} \, \frac{1}{n^{2\alpha}} \right]^{d-\#I_d} - 1 \right)^{1/2} \\
		& \qquad\quad \qquad \qquad \qquad \qquad \qquad \times \left(1 + 2 \sum_{\ell = 1}^{\#I_d} \left[ \frac{\beta_1\, \NR{\alpha\ell}^{1/\ell}}{\beta_0} \, \frac{1}{n^{2\alpha}} \right]^{\ell} \right)^{1/2}
	\end{align*}
	if $\# I_d<d$, and
	\begin{equation*}
		E(Q_n(\bsz))		
		\geq e(0,d; \SI_{I_d}(F_d(r_{\alpha,\bm{\beta}}))) \left(2 \sum_{\ell = 1}^{d} \left[ \frac{\beta_1\,\NR{\alpha\ell}^{1/\ell}}{\beta_0} \, \frac{1}{n^{2\alpha}} \right]^{\ell} \right)^{1/2}
	\end{equation*}
	if $\# I_d=d$. In particular,
	\begin{equation*}
		E(Q_n(\bsz))
		\geq c \, \max\{d-\#I_d,1\}^{1/2} \, n^{-\alpha},
	\end{equation*}
	where $c=\sqrt{2\beta_1 \NR{\alpha}/\beta_0}$ does not depend on $d$ and $n$.
	\end{itemize}
\end{theorem}

\begin{proof}
Let $Q_n(\bsz)$ be given. 
From\RefEq{eq:KshinvZd} we obtain
\begin{align}
    \label{eq:mse2}
	E(Q_n(\bsz))^2 
    &=
    \sum_{\bszero \ne \bsh \in L^{\bot}} \frac{r^{-1}_{\alpha,\bm{\beta}}(\bsh)}{\#\S_d} \, \M_d(\bsh)!
    \\
    \nonumber
    &\le
    \sum_{\bszero \ne \bsh \in L^{\bot}} \frac{r^{-1}_{\alpha,\bm{\beta}}(\bsh)}{\#\S_d} \, \sum_{P \in S_d} \ind{P(\bsh) \in L^{\bot}}
    =
    e^{\mathrm{wor}}(Q_{n}(\bsz); \SI_{I_d}(F_d(r_{\alpha,\bm{\beta}})))^2
    ,
\end{align}
where the last line is the squared worst case error for the unshifted lattice rule from \RefProp{prop:latticeerror}. The inequality holds since, by definition, $\M_d(\bsh)!$ is the number of $P\in\S_d$ such that $P(\bsh)=\bsh$ and we sum over all $\bsh \in L^{\bot}$.
Due to the mean value property, there clearly exists a shift $\bsDelta^*\in[0,1)^d$ such that
\begin{equation*}
	e^{\mathrm{wor}}(Q_{n}(\bsz) + \bsDelta^*; \SI_{I_d}(F_d(r_{\alpha,\bm{\beta}})))
	\leq E(Q_n(\bsz)).
\end{equation*}

To prove the lower bounds 
we again use the fact that $n\Z^d\subseteq L^{\bot}$. 
To this end, we first consider the case $I_d \subsetneq \{1,\ldots,d\}$, i.e., $I_d^c = \{1,\ldots,d\} \setminus I_d  \neq\emptyset$. By $\J_d$ we denote the set of all indices $\bsh=(h_1,\ldots,h_d)\in\Z^d$ such that, for some $\setu \subseteq I_d$,
\begin{equation*}
	\bsh\big|_{I_d^c} = n \bsk
	\qquad \text{and}  \qquad
	h_j = \begin{cases}
		n h, & \text{if } j \in \setu, \\
		0, & \text{if } j\in I_d \setminus \setu
	\end{cases}
\end{equation*}
for $\bsk\in \Z^{d-\# I_d} \setminus \{\bm{0}\}$ and $h\in \Z \setminus \{0\}$.
By construction $\J_d \subseteq L^{\bot} \setminus \{\bm{0}\}$ and for all $\bsh\in\J_d$ we have $\M_d(\bsh)! = (\#\setu)!\,(\#I_d-\#\setu)!$,  as well as
\begin{align*}
	r^{-1}_{\alpha,\bm{\beta}}(\bsh) 
	&= r^{-1}_{\alpha,\bm{\beta}}(n\bsk) \, \beta_0^{\# I_d - \# \setu} \, \beta_1^{\#\setu} \, R(n\abs{h})^{-2\alpha\#\setu} \\
	&= \beta_0^{\# I_d} \, r^{-1}_{\alpha,\bm{\beta}}(n\bsk) \, \left[ \frac{\beta_1}{\beta_0} \right]^{\#\setu}R(n\abs{h})^{-2\alpha\#\setu}. 
\end{align*}
Thus, \link{eq:mse2} implies
\begin{align*}
	E(Q_n(\bsz))^2 
	&\geq \sum_{\substack{\bsh \in \J_d \\ (\setu \subseteq I_d)}} \frac{\M_d(\bsh)!}{\#\S_d} \, r^{-1}_{\alpha,\bm{\beta}}(\bsh)\\
	&= \sum_{\substack{\bm{0} \neq \bsk\in \Z^{d-\# I_d} \\ (\setu = \emptyset)}} \frac{(\#I_d)!}{(\#I_d)!} \, \beta_0^{\# I_d} \, r^{-1}_{\alpha,\bm{\beta}}(n\bsk) \\
		&\quad\qquad + \sum_{\ell = 1}^{\#I_d} \sum_{\substack{\setu\subseteq I_d\\\#\setu=\ell}} \sum_{\bm{0} \neq \bsk\in \Z^{d-\# I_d}} \sum_{0 \neq h \in \Z} \frac{\ell!\,(\#I_d-\ell)!}{(\#I_d)!} \, \beta_0^{\# I_d} \, r^{-1}_{\alpha,\bm{\beta}}(n\bsk) \, \left[ \frac{\beta_1}{\beta_0} \right]^{\ell} R(n\abs{h})^{-2\alpha\ell} \\
	&= \beta_0^{\# I_d} \left(\sum_{\bm{0} \neq \bsk\in \Z^{d-\# I_d}} r^{-1}_{\alpha,\bm{\beta}}(n\bsk) \right) \left(1 + \sum_{\ell = 1}^{\#I_d} \sum_{\substack{\setu\subseteq I_d\\\#\setu=\ell}} \frac{1}{\binom{\#I_d}{\ell}} \! \left[ \frac{\beta_1}{\beta_0} \right]^{\ell} \! \sum_{0 \neq h \in \Z} \! R(n \abs{h})^{-2\alpha\ell} \right).
\end{align*}
Similar to the proof of \autoref{thm:unshifted} we estimate
\begin{align*}
	\sum_{\bm{0} \neq \bsk\in \Z^{d-\# I_d}} r^{-1}_{\alpha,\bm{\beta}}(n\bsk)
	&\geq \beta_0^{d-\#I_d} \left( \left[ 1 + \frac{2\beta_1 \NR{\alpha}}{\beta_0} \, \frac{1}{n^{2\alpha}} \right]^{d-\#I_d} - 1 \right),
\end{align*}
as well as
\begin{equation*}
	\sum_{0\neq h \in \Z} R(n \abs{h})^{-2\alpha\ell}
	\geq 2 \, n^{-2\alpha\ell} \sum_{m=1}^\infty R(m)^{-2\alpha\ell}
	= 2 \left[ \NR{\alpha\ell}^{1/\ell} \, \frac{1}{n^{2\alpha}} \right]^{\ell}
\end{equation*}
for $\ell=1,\ldots,\#I_d$.
Since $\# \{ \setu \subseteq I_d \sep \#\setu = \ell \} = \binom{\#I_d}{\ell}$ for those $\ell$, it follows
\begin{align*}
	E(Q_n(\bsz))^2 
	\geq \beta_0^{d} \left( \left[ 1 + \frac{2\beta_1 \NR{\alpha}}{\beta_0} \, \frac{1}{n^{2\alpha}} \right]^{d-\#I_d} - 1 \right) \left(1 + 2 \sum_{\ell = 1}^{\#I_d} \left[ \frac{\beta_1\,\NR{\alpha\ell}^{1/\ell}}{\beta_0} \, \frac{1}{n^{2\alpha}} \right]^{\ell} \right).
\end{align*}

The lower bound for the case $I_d=\{1,\ldots,d\}$, i.e., $\# I_d=d$, can be derived similarly but then we need to exclude $\setu=\emptyset$ in order to ensure $\bm{0}\notin\J_d$.
Finally, we use Bernoulli's inequality (see \RefRem{rem:Bernoulli}) and the fact that $\beta_0^d$ equals the squared initial error $e(0,d; \SI_{I_d}(F_d(r_{\alpha,\bm{\beta}})))^2$ to complete the proof.
\end{proof}

In order to show the existence of good shifted lattice rules, we are left with finding generating vectors $\bsz\in\Z_n^d$ such that $E(Q_n(\bsz))$ is upper bounded appropriately. 
In view of \autoref{thm:rmse} the best rate of convergence we can hope for is $n^{-\alpha}$ and the constants will be independent of the dimension~$d$ only if $(d-\# I_d) \in \0(1)$. 
Moreover, it is known that already for $d=1$ this rate cannot be improved. 
We refer to~\cite{SJ94} for details.

To derive the desired existence result we need a lemma which is based on the character property \link{eq:char}. 
For its proof we refer to the appendix (\RefSec{sect:appendix}).
\begin{lemma}\label{lem:char_prop}
	Let $d\in\N$, $\bsh \in \Z^d$, and $n\in\N$ prime. 
	Then
	\begin{align*}
		\frac{1}{\#\Z_n^d} \sum_{\bsz\in\Z_n^d} \ind{\bsh\in L(\bsz)^\bot}
		&= \frac{1}{n} \sum_{j=0}^{n-1} \prod_{\ell=1}^{d} \frac{1}{n} \sum_{z_\ell=0}^{n-1} \exp(\twopii \, j h_\ell z_\ell / n) 
		= \begin{cases}
		    1 & \text{if } \bsh \equiv \bszero \pmod{n}, \\
		   \displaystyle
		    n^{-1}, & \text{otherwise},
		  \end{cases} 
	\end{align*}
	where $L(\bsz)^\bot$ denotes the dual lattice induced by $\bsz$ and $\bsh \equiv \bszero \pmod{n}$ is a shorthand for $h_\ell \equiv 0 \pmod{n}$ for all $1 \le \ell \le d$.
\end{lemma}

Now we are ready to establish the main result of this paper, the existence of shifted rank-$1$ lattice rules which nearly achieve $\0(n^{-\alpha})$ convergence for numerical integration of $I_d$-permutation-invariant functions. 
To this end we prove that for carefully chosen generating vectors the root mean squared worst case error decays with a rate arbitrarily close to~$\alpha$. 
For explicit component-by-component constructions of such generating vectors we refer to the forthcoming paper \cite{NSW14}.
\begin{theorem}\label{thm:existence}
	Let $d\in\N$, $I_d\subseteq\{1,\ldots,d\}$, and $n\in\N$ with $n\geq c_R$ be prime. 
	Then there exists a generating vector $\bsz^*\in \Z_n^d$ such that the mean squared worst case error of $Q_n(\bsz^*)+\bsDelta$ w.r.t.\ all shifts $\bsDelta\in[0,1)^d$ satisfies
	\begin{align*}
		E(Q_n(\bsz^*))^2
		\leq (1+c_R)^{\lambda}\, C_{d,\lambda}(r_{\alpha,\bm{\beta}}) \, \frac{1}{n^{\lambda}}
		\qquad \text{for all} \qquad 1 \leq \lambda < 2\alpha
	\end{align*}
	with $c_R$ as defined in \autoref{subsect:subspaces} and
	\begin{align}
        \label{C_dlambda}
		C_{d,\lambda}(r_{\alpha,\bm{\beta}})
		= \left( \sum_{\bm{0}\neq \bsh \in \Z^d} \left[ \frac{\M_d(\bsh)!}{\# \S_d} \, r^{-1}_{\alpha,\bm{\beta}}(\bsh) \right]^{1/\lambda} \right)^{\lambda}.
	\end{align}
\end{theorem}

\begin{proof}
For the optimal choice $\bsz^* \in \Z_n^d$ which minimizes the mean squared worst case error \link{eq:mse} and all $\lambda>0$ we naturally have $E(Q_n(\bsz^*))^{2/\lambda} \leq E(Q_n(\bsz))^{2/\lambda}$
for every $\bsz\in\Z_n^d$, i.e.,
\begin{equation*}
	E(Q_n(\bsz^*))^{2/\lambda}
	\leq \frac{1}{\# \Z_n^d} \sum_{\bsz\in\Z_n^d} 	E(Q_n(\bsz))^{2/\lambda} .
\end{equation*}
We now use \link{eq:mse2} to expand $E(Q_n(\bsz))^2$, $\bsz\in\Z_n^d$, and apply Jensen's inequality (see \RefLem{lem:jensens} in the appendix) for $p=1\geq 1/\lambda = q$ to obtain
\begin{align*}
	E(Q_n(\bsz))^{2/\lambda}
	&= \left( \sum_{\bszero \ne \bsh \in L(\bsz)^{\bot}} \frac{\M_d(\bsh)!}{\# \S_d} \, r^{-1}_{\alpha,\bm{\beta}}(\bsh) \right)^{1/\lambda} 
	\leq \sum_{\bszero \ne \bsh \in L(\bsz)^{\bot}} \left[ \frac{\M_d(\bsh)!}{\# \S_d} \, r^{-1}_{\alpha,\bm{\beta}}(\bsh) \right]^{1/\lambda} 
\end{align*}
for all $\bsz\in\Z_n^d$.
Combining both estimates yields
\begin{align*}
	E(Q_n(\bsz^*))^{2/\lambda}
	&\leq  \frac{1}{\# \Z_n^d} \sum_{\bsz\in\Z_n^d} \sum_{\bszero \ne \bsh \in L(\bsz)^{\bot}} \left[ \frac{\M_d(\bsh)!}{\# \S_d} \, r^{-1}_{\alpha,\bm{\beta}}(\bsh) \right]^{1/\lambda} \\
	&= \sum_{\bszero \neq \bsh \in \Z^d} \left[ \frac{\M_d(\bsh)!}{\# \S_d} \, r^{-1}_{\alpha,\bm{\beta}}(\bsh) \right]^{1/\lambda} \frac{1}{\# \Z_n^d} \sum_{\bsz\in\Z_n^d} \ind{\bsh\in L(\bsz)^\bot} .
\end{align*}
From \RefLem{lem:char_prop} we derive
\begin{align*}
	E(Q_n(\bsz^*))^{2/\lambda} 
	&\leq \sum_{\substack{\bszero \neq \bsh \in \Z^d\\ \bsh \equiv \bszero \tpmod{n}}} \!\left[ \frac{\M_d(\bsh)!}{\# \S_d} \, r^{-1}_{\alpha,\bm{\beta}}(\bsh) \right]^{1/\lambda} + \frac{1}{n} \! \sum_{\substack{\bszero \neq \bsh \in \Z^d\\\exists \ell\colon h_\ell \not\equiv 0 \tpmod{n}}} \!\left[ \frac{\M_d(\bsh)!}{\# \S_d} \, r^{-1}_{\alpha,\bm{\beta}}(\bsh) \right]^{1/\lambda} \\
	&\leq \left(\frac{c_R}{n}\right)^{2\alpha/\lambda} \sum_{\bszero \neq \bsh \in \Z^d} \!\left[ \frac{\M_d(\bsh)!}{\# \S_d} \, r^{-1}_{\alpha,\bm{\beta}}(\bsh) \right]^{1/\lambda} + \frac{1}{n} \sum_{\bszero \neq \bsh \in \Z^d} \!\left[ \frac{\M_d(\bsh)!}{\# \S_d} \, r^{-1}_{\alpha,\bm{\beta}}(\bsh) \right]^{1/\lambda} \\
	&\leq \frac{1+c_R}{n} \sum_{\bszero \neq \bsh \in \Z^d} \!\left[ \frac{\M_d(\bsh)!}{\# \S_d} \, r^{-1}_{\alpha,\bm{\beta}}(\bsh) \right]^{1/\lambda},
\end{align*}
where we used that for all $\bsh=n\bsk\in n\Z^d\setminus\{\bm{0}\}$ it is
\begin{equation*}
	\M_d(n\bsk)! = \M_d(\bsk)! 
	\quad \text{and} \quad
	r^{-1}_{\alpha,\bm{\beta}}(n\bsk)
	\leq \left( \frac{c_R}{n} \right)^{2\alpha\abs{\bsk}_0} \, r^{-1}_{\alpha,\bm{\beta}}(\bsk)
	\leq \left( \frac{c_R}{n} \right)^{2\alpha} \, r^{-1}_{\alpha,\bm{\beta}}(\bsk),
\end{equation*}
since we assumed that $\lambda < 2\alpha$ as well as $n \geq c_R$.
\end{proof}

As already stated in the introduction, not only the rate of convergence but also the dependence of the error bounds on the dimension $d$ plays an important role in modern research and computational practice. As we will see in \RefProp{prop:constant} below, for fixed $\bm{\beta}=(\beta_0,\beta_1)$, the constant $C_{d,\lambda}(r_{\alpha,\bm{\beta}})$ in the estimate stated in \autoref{thm:existence} can be bounded polynomially in $d$ only if we restrict ourselves to the case $\lambda=1$ which corresponds to the Monte Carlo rate of convergence $n^{-1/2}$. Furthermore, even in this case we need to assume reasonably small parameters $\beta_1$, as well as enough permutation-invariance conditions. In detail, we need
\begin{equation*}
	(d-\#I_d)\in \0(\ln d)
	\qquad \text{and} \qquad
	\frac{\beta_1}{\beta_0 R(m)^{2\alpha}} < 1 \quad \text{for all} \quad m\in\N
\end{equation*}
in order to avoid an exponential growth with the dimension.
The proofs of the following assertions are postponed to the appendix (\RefSec{sect:appendix}).

\begin{proposition}\label{prop:constant}
For $d\in\N$, $I_d\subseteq\{1,\ldots,d\}$, $r_{\alpha,\bm{\beta}}$ as in \RefSec{sect:setting}, and $\lambda \geq 1$ consider the constant $C_{d,\lambda}(r_{\alpha,\bm{\beta}})$ defined by \link{C_dlambda}. Then
\begin{itemize}
	\item[$\bullet$] $C_{d,\lambda}(r_{\alpha,\bm{\beta}})$ is a monotonically increasing, continuous function of $\lambda$, i.e.,
		\begin{equation*}
			C_{d,\lambda}(r_{\alpha,\bm{\beta}}) 
			\leq C_{d,\mu}(r_{\alpha,\bm{\beta}})
			\qquad \text{for all} \qquad 1 \leq \lambda \leq \mu.
		\end{equation*}
	Moreover, for all $\lambda\geq 1$ this constant scales with the squared initial error. That is,
		\begin{equation}\label{eq:scaling}
			C_{d,\lambda}(r_{\alpha,\bm{\beta}}) 
			= e(0,d; \SI_{I_d}(F_d(r_{\alpha,\bm{\beta}})))^2 \, C_{d,\lambda}(r_{\alpha,(1,\beta_1/\beta_0)}).
		\end{equation}
	\item[$\bullet$] For $\lambda=1$ and all $m\in\N$ we have
		\begin{align}
			C_{d,1}(r_{\alpha,\bm{\beta}}) 
			&= e(0,d; \SI_{I_d}(F_d(r_{\alpha,\bm{\beta}})))^2 \, \left( \frac{M_{2,d}(K_{d,I_d})}{S_d(K_{d,I_d})} - 1 \right) \label{eq:repres}\\
			&\geq e(0,d; \SI_{I_d}(F_d(r_{\alpha,\bm{\beta}})))^2 \, \left( 2\, \left[ \frac{\beta_1}{\beta_0 \, R(m)^{2\alpha}} \right]^d -1 \right),\nonumber
		\end{align}
		where $M_{2,d}(K_{d,I_d})$ and $S_d(K_{d,I_d})$ are given by \RefLem{lem:SdMd} and $K_{d,I_d}$ denotes the reproducing kernel of $\SI_{I_d}(F_d(r_{\alpha,\bm{\beta}}))$.
	\item[$\bullet$] The constant $C_{d,\lambda}(r_{\alpha,\bm{\beta}})$ is lower bounded as follows: 
	In the fully permutation-invariant case ($\# I_d=d$) it holds
		\begin{align*}
			&C_{d,\lambda}(r_{\alpha,\bm{\beta}}) \geq e(0,d; \SI_{I_d}(F_d(r_{\alpha,\bm{\beta}})))^2 
			\cdot \begin{cases}
				2\, \sum_{\ell=1}^d\limits \left( \frac{\beta_1}{\beta_0 R(1)^{2\alpha}}\right)^\ell &  \text{ if} \quad  \lambda = 1, \\
				2^\lambda \left[ \left( 1 + \left[ \frac{\beta_1}{\beta_0\, R(1)^{2\alpha}} \right]^{1/(\lambda-1)} \right)^{d} -1 \right]^{\lambda-1} & \text{ if} \quad \lambda > 1,
			\end{cases}
		\end{align*}
		whereas in the case $\# I_d < d$ we have
		\begin{align}
			C_{d,\lambda}(r_{\alpha,\bm{\beta}}) 
			& \geq e(0,d; \SI_{I_d}(F_d(r_{\alpha,\bm{\beta}})))^2 \left[ \left( 1+ 2 \left[ \frac{\beta_1}{\beta_0}\right]^{1/\lambda} \NR{\alpha/\lambda}\right)^{d-\#I_d}-1 \right]^\lambda \label{est:lower_bound}\\
			& \qquad \quad \times
			\begin{cases}
				\left( 1 + 2\, \sum_{\ell=1}^{\#I_d} \limits \left[ \frac{\beta_1}{\beta_0 R(1)^{2\alpha}}\right]^\ell\right) &  \text{ if} \quad  \lambda = 1, \\[4mm]
				\left( 1+ 2^\lambda \left[ \left( 1+ \left[ \frac{\beta_1}{\beta_0\, R(1)^{2\alpha}}\right]^{1/(\lambda-1)} \right)^{\#I_d} -1 \right]^{\lambda-1}\right) & \text{ if} \quad \lambda > 1.
			\end{cases} \nonumber
		\end{align}
		(If we do not have any permutation-invariance, i.e., if $I_d=\emptyset$, then the lower bound reduces to the first line \link{est:lower_bound} with $d-\# I_d$ replaced by $d$.)
		\item[$\bullet$] Finally, if $1<\lambda < 2\alpha$ and $A>0$ is chosen such that $\alpha > A+1/2 > \lambda/2$,
	then for all $\gamma > 0$ there holds the upper bound
	\begin{align}\label{est:upper_bound}
		C_{d,\lambda}(r_{\alpha,(\beta_0,\beta_1)}) 
		\leq C_{d,1}(r_{\alpha-A,(\beta_0, \beta_1 \gamma)})  \left( \left[ 1 + 2 \, \gamma^{-1/(\lambda-1)} \, \NR{A/(\lambda-1)} \right]^d - 1 \right)^{\lambda-1} < \infty.
	\end{align}	
\end{itemize}
\end{proposition}
\begin{remark}
If we allow weight parameters $\beta_1$ which decay with the dimension $d$ then~\link{est:upper_bound} can be used to bound $C_{d,\lambda}(r_{\alpha,(\beta_0,\beta_1)})$ polynomially in $d$ also for $\lambda>1$.
To this end, let us record that
\begin{align*}
	\left[ 1 + 2 \, \gamma^{-1/(\lambda-1)} \, \NR{A/(\lambda-1)} \right]^d
	&\leq \exp\left(2 \, \NR{A/(\lambda-1)} \frac{d}{\gamma^{1/(\lambda-1)}} \right)
\end{align*}
is polynomially upper bounded in $d$ if $\gamma=\gamma(d)$ is chosen such that
\begin{equation*}
	\gamma \geq C\,\left( \frac{d}{\ln (d+1)} \right)^{\lambda-1}
\end{equation*}
for some constant $C>0$ and all $d\in\N$.
In order to bound the first factor in \link{est:lower_bound} we follow the lines of  \autoref{sect:avgbounds}. Thus, it is sufficient to ensure that for all $m\in\N$
\begin{equation*}
	\frac{2\, \beta_1\gamma}{\beta_0 R(m)^{2(\alpha-A)}} \leq 1
	\qquad \text{and} \qquad
	\beta_1 \gamma \leq C \frac{\ln (d+1)}{\max\{d-\#I_d,1\}},
\end{equation*}
see \RefRem{rem:decaying_weights}. 
Choosing $\gamma$ as above then leads to the condition
\begin{equation*}
	\beta_1=\beta_1(d) 
	\leq \frac{\left[ \ln (d+1) \right]^{\lambda}}{d^{\lambda-1}} \frac{c}{\max\{\ln(d+1),d-\#I_d\} }
\end{equation*}
which generalizes the condition for $\lambda=1$.
\end{remark}

Recall that for fixed parameters $r_{\alpha,\bm{\beta}}$ and fixed dimension $d$ the blowup of the constant $C_{d,\lambda}(r_{\alpha,\bm{\beta}})$ for $\lambda>1$ is quite typical. 
Therefore the case $\lambda=1$ deserves special attention.
We summarize the final assertion for this case in the next corollary.

\begin{corollary}\label{cor:existence}
	Let $d\in\N$ and $I_d\subseteq\{1,\ldots,d\}$. 
	Then for all $n\in\N$ prime with $n\geq c_R$ there exists a shifted rank-$1$ lattice rule $Q_n(\bsz^*)+\bsDelta^*$ for integration of $I_d$-permutation-invariant functions in $F_d(r_{\alpha,\bm{\beta}})$ such that
	\begin{align}\label{opt_onehalf_bound}
		&e^{\mathrm{wor}}(Q_{n}(\bsz^*)+\bsDelta^*; \SI_{I_d}(F_d(r_{\alpha,\bm{\beta}}))) \nonumber\\
		&\qquad\qquad\leq \sqrt{1+c_R} \, \sqrt{\frac{M_{2,d}(K_{d,I_d})}{S_d(K_{d,I_d})} - 1} \; n^{-1/2} \; e(0,d; \SI_{I_d}(F_d(r_{\alpha,\bm{\beta}}))).
	\end{align}
	Therefore, up to some small constant, it realizes the bounds stated in \RefProp{prop:avg_bounds} and \RefThm{thm:tractability}, respectively. Consequently, our tractability results can be achieved using shifted rank-$1$ lattice rules.
\end{corollary}
\begin{proof}
Due to \autoref{thm:existence} (for $\lambda=1$) there exists a generating vector $\bsz^*\in\Z_n^d$ such that
\begin{equation*}
	E(Q_n(\bsz^*)) 
	\leq \sqrt{1+c_R} \sqrt{C_{d,1}(r_{\alpha,\bm{\beta}})}\, n^{-1/2}.
\end{equation*}
Moreover, the mean value property implies the existence of some $\bsDelta^*=\bsDelta^*(\bsz^*)\in[0,1)^d$ with
\begin{equation*}
	e^{\mathrm{wor}}(Q_{n}(\bsz^*)+\bsDelta^*; \SI_{I_d}(F_d(r_{\alpha,\bm{\beta}}))) 
	\leq E(Q_n(\bsz^*)).
\end{equation*}
Consequently, \link{eq:repres} in \RefProp{prop:constant} yields the claim.
\end{proof}

\begin{remark}\label{rem:constant}
Note that more elaborate estimates in the proof of \autoref{thm:existence} allow to reduce the constant $1+c_R$ to $1+\delta$ with arbitrarily small $\delta>0$ when we assume that $n$ is larger than some constant only depending on $\alpha$, $\lambda$, $c_R$, and $\delta$. This clearly effects the bound \link{opt_onehalf_bound} in \RefCol{cor:existence} as well.
\end{remark}


\section{Appendix}\label{sect:appendix}
In this final section we collect the proofs of all lemmas and propositions we postponed in the course of this paper.

\subsection*{Proof of \RefLem{lem:SdMd}}
\begin{proof}
Due to the definition of $M_{1,d}$, $M_{2,d}$, and $S_d$ 
we can restrict ourselves to the study of the extremal cases of the
fully permutation-invariant spaces $\SI(F_d(r_{\alpha,\bm{\beta}}))$ 
and the spaces $F_d(r_{\alpha,\bm{\beta}})$ without any permutation-invariance
since \link{tensor_kernel} implies
\begin{gather*}
		X_d(K_{d,I_d}) = X_{\#I_d}(K_{\#I_d,\{1,\ldots,\#I_d\}}) \; X_{d-\#I_d}(K_{d-\#I_d})
\end{gather*}
for $X \in\{S, M_{1}, M_{2}\}$.

For the fully permutation-invariant spaces $H_d = \SI(F_d(r_{\alpha,\bm{\beta}}))$
induced by $K=K_{d,\{1,\ldots,d\}}$, as well as for the entire spaces $H_d=F_d(r_{\alpha,\bm{\beta}})$, where $K=K_d$, the initial error for integration clearly equals $\beta_0^{d/2}$. This proves $S_d(K_{d,I_d})=\beta_0^d$.

We turn to the derivation of $M_{2,d}$. 
If $H_d=F_d(r_{\alpha,\bm{\beta}})$ then
\begin{align*}
		M_{2,d}(K_d) 
		= \left( \int_0^1 K_1(x,x) \, \rd x \right)^d 
		= \left( \sum_{h\in\Z} r^{-1}_{\alpha,\bm{\beta}}(h) \right)^d 
		= \sum_{\bsh\in\Z^d} r^{-1}_{\alpha,\bm{\beta}}(\bsh)
\end{align*}
due to the (tensor) product structure of the objects involved.
Using the definition of $r_{\alpha,\bm{\beta}}$ for $d=1$ we see that
\begin{align*}
		\sum_{h\in\Z} r^{-1}_{\alpha,\bm{\beta}}(h)
		= \sum_{h\in\Z} \left(\delta_{0,h} \beta_0 + (1-\delta_{0,h}) \beta_1 R(\abs{h})^{-2\alpha}\right)
		&= \beta_0 + 2 \beta_1 \sum_{m\in\N} R(m)^{-2\alpha} 
		= \beta_0 + 2 \beta_1 \NR{\alpha},
\end{align*}
where $\NR{\alpha}$ is given by \link{summableR}.
Thus, we have shown that 
\begin{equation*}
	M_{2,d}(K_d)
	= \beta_0^d \, \left( 1 + \frac{2 \beta_1 \NR{\alpha}}{\beta_0} \right)^d.
\end{equation*}
Since $K_1(x,x)$ is constant with respect to $x\in[0,1]$ we see that $M_{1,d}(K_d)= M_{2,d}(K_d)$ which finally implies \link{eq:M1}.

For the fully permutation-invariant case we need a little more effort. We restrict ourselves to $M_{2,d}$.
In this case,
\begin{align*}
 		M_{2,d}(K_{d,\{1,\ldots,d\}}) 
    &= \sum_{\bsk \in \nabla_d} \frac{r^{-1}_{\alpha,\bm{\beta}}(\bsk)}{\M_d(\bsk)!} \sum_{P \in \S_d} \int_{[0,1]^d} \exp(\twopii (\bsk-P(\bsk)) \cdot \bsx) \rd\bsx 
    .
\end{align*}
It is clear that the integral is~$1$ whenever $\bsk = P(\bsk)$, which happens exactly $\M_d(\bsk)!$ times out of all $P \in \S_d$, and $0$~otherwise.
Thus, we get
\begin{align*}
	M_{2,d}(K_{d,\{1,\ldots,d\}})
    = \sum_{\bsk \in \nabla_d} r^{-1}_{\alpha,\bm{\beta}}(\bsk). 
\end{align*}
Now the (tensor) product structure of the set $\nabla_d$, see \link{def_nabla}, the weights $r_{\alpha,\bm{\beta}}$, the kernel $K_{d,I_d}$, see \link{tensor_kernel}, and the quantity $M_{2,d}$, implies that the latter expression remains valid for $K=K_{d,I_d}$ with arbitrary subsets $\emptyset \neq I_d\subseteq \{1,\ldots,d\}$.
Finally, note that for $\#I_d<2$ we have $\nabla_d=\Z^d$ which completes the proof.
\end{proof}

\subsection*{Proof of \RefLem{lemmaBound}}
\begin{proof}
We first note the following equality
\begin{gather}\label{case_mNEW}
				\sum_{\substack{\bm{k}\in\N_0^s \\m\leq k_1\leq\cdots\leq k_s}} \lambda_{s,\bm{k}}
				= \lambda_m^s + \sum_{\ell=1}^s \lambda_m^{s-\ell}\sum_{ \substack{ \bm{j}\in\N^\ell\\m+1\leq j_1\leq\cdots\leq j_\ell } } \lambda_{\ell,\bm{j}}
				\qquad \text{for all} \quad s\in\N,
\end{gather}
which follows by considering $\ell$ of the $k_j$'s to be larger than $m$ and by the product structure of $\lambda_{s,\bsk}$.
We now prove \link{estimate_VNEW} via induction on $V\in\N_0$.
Therefore, let $d\in\N$ be fixed arbitrarily.
Setting $s=d$ and $m=0$ in\RefEq{case_mNEW} corresponds to\RefEq{estimate_VNEW} with $V=0$.
Thus, assume \link{estimate_VNEW} to be true for some fixed $V\in\N_0$.
Then, by using \link{case_mNEW} for $s=L$ and $m=V+1$, we see that the right hand side of \link{estimate_VNEW} equals 
\begin{align*}
  & \lambda_0^d \, d^V \Bigg( 1 + V + \sum_{L=1}^d \lambda_0^{-L} \Bigg( \lambda_{V+1}^L + \sum_{\ell=1}^L \lambda_{V+1}^{L-\ell} \sum_{ \substack{ \bm{j}\in\N^\ell \\V+2\leq j_1\leq\cdots\leq j_\ell } } \lambda_{\ell,\bm{j}} \Bigg) \Bigg) \\
  &\qquad\qquad = \lambda_0^d \, d^V \Bigg( 1 + V + \sum_{L=1}^d \left(\frac{\lambda_{V+1}}{\lambda_0}\right)^{L} 
  + \sum_{L=1}^d \sum_{\ell=1}^L \left(\frac{\lambda_{V+1}}{\lambda_0}\right)^{L-\ell} \lambda_0^{-\ell}  \sum_{ \substack{ \bm{j}\in\N^\ell\\(V+1)+1\leq j_1\leq\cdots\leq j_\ell } } \lambda_{\ell,\bm{j}} \Bigg). 
\end{align*}
We now decouple the double sum by letting $\ell$ go up to $d$.
The sums on $L$ can then be bounded by $d$ as (by assumption) we have $\lambda_{V+1}/\lambda_0 \le 1$.
Now also bounding $1+V \le d\,(1+V)$ we obtain
\begin{align*}
					\sum_{\substack{\bm{k}\in\N_0^d\\0\leq k_1\leq\cdots\leq k_d}} \lambda_{d,\bm{k}}
					&\leq \lambda_0^{d} \, d^{V+1} \left( 1 + (V + 1) + \sum_{\ell=1}^d \lambda_0^{-\ell} \sum_{ \substack{ \bm{j}\in\N^\ell\\(V+1)+1\leq j_1\leq\cdots\leq j_\ell } } \lambda_{\ell,\bm{j}} \right)
\end{align*}
which completes the induction step.
\end{proof}

\subsection*{Proof of \RefLem{lem:e2}}
\begin{proof}
Using\RefEq{sym_kernel_2} we obtain
\begin{align}
	\int_{[0,1]^d} K_{d,I_d}(\bsx, \bsy) \rd\bsx 
	&= \sum_{P \in \S_d} \sum_{\bsh \in \Z^d} \frac{r^{-1}_{\alpha,\bm{\beta}}(\bsh)}{\#\S_d} \exp\left(-\twopii\bsh \cdot \bsy \right) \int_{[0,1]^d} \exp \left(\twopii \bsh \cdot P(\bsx) \right) \rd\bsx \nonumber\\
	&= r^{-1}_{\alpha,\bm{\beta}}(\bm 0) \label{eq:vanish_intK}
\end{align}
independent of $\bsy\in[0,1]^d$ since the last integral equals one for $\bsh =\bm{0}$ and zero otherwise.
Therefore
\begin{align*}
  \int_{[0,1]^d} \int_{[0,1]^d} K_{d,I_d}(\bsx, \bsy) \rd\bsx \rd{\bsy}
  &=
  r^{-1}_{\alpha,\bm{\beta}}(\bm 0)
\end{align*}
and
\begin{align*}
  - \frac2n \sum_{j=0}^{n-1} w_j \int_{[0,1]^d} K_{d,I_d}(\bsx, \bst^{(j)}) \rd\bsx 
  &= - \, r^{-1}_{\alpha,\bm{\beta}}(\bm 0) \, \frac2n \sum_{j=0}^{n-1} w_j.
\end{align*}
The remaining term in~\eqref{eq:e2-K} is the double cubature sum for which we obtain
\begin{align*}
	\frac1{n^2} \sum_{j,\ell=0}^{n-1} w_j\, w_\ell \, K_{d,I_d}(\bst^{(j)}, \bst^{(\ell)}) 
	 &= \sum_{\bsh \in \Z^d} r^{-1}_{\alpha,\bm{\beta}}(\bsh) \left( \frac1{n} \sum_{\ell=0}^{n-1} w_\ell\,  \exp(-\twopii \bsh\cdot \bst^{(\ell)})\right) \\
	&\qquad\qquad\qquad \times \left( \frac{1}{\#\S_d} \sum_{P \in \S_d} \frac1{n} \sum_{j=0}^{n-1} w_j \,\exp(\twopii \bsh\cdot P(\bst^{(j)})) \right) 
\end{align*}
which directly follows from\RefEq{sym_kernel_2}.
Summing up the three contributions and replacing $P$ by $P^{-1}$ now proves the claim.
\end{proof}

\subsection*{Proof of \RefLem{lem:char_prop}}
\begin{proof}
The first equality in the statement of \RefLem{lem:char_prop} follows from the character property~\link{eq:char} and $\Z_n=\{0,1,\ldots,n-1\}$ since
\begin{align*}
	\frac{1}{\#\Z_n^d} \sum_{\bsz\in\Z_n^d} \ind{\bsh\in L(\bsz)^\bot}
	&= \frac{1}{n^d} \sum_{z_1\in\Z_n}\cdots\sum_{z_d\in\Z_n} \frac1n \sum_{j=0}^{n-1} \exp\!\left(\twopii \, \frac{j}{n} \, \sum_{\ell=1}^d h_\ell z_\ell\right)
	\\
	&=
	\frac{1}{n} \sum_{j=0}^{n-1} \prod_{\ell=1}^{d} \frac{1}{n} \sum_{z_\ell=0}^{n-1} \exp(\twopii \, h_\ell (j z_\ell) / n) 
	.
\end{align*}
For $j=0$ we have
\begin{align*}
  \prod_{\ell=1}^{d} \frac{1}{n} \sum_{z_\ell=0}^{n-1} \exp(0)
  &=
  1,
\end{align*}
while for $j\ne 0$ and $n$ prime we have that $j \Z_n = \Z_n$ and thus for each $0 < j < n$ it holds
\begin{align*}
  \prod_{\ell=1}^d \frac{1}{n} \sum_{z_\ell=0}^{n-1} \exp(\twopii \, h_\ell z_\ell / n) 
  &=
  \prod_{\ell=1}^d \begin{cases} 
    1, & \text{if $h_j \equiv 0 \pmod{n}$}, \\
    0, & \text{otherwise}
  \end{cases}
  \\
  &=
  \begin{cases}
    1, & \text{if all $h_j \equiv 0 \pmod{n}$}, \\
    0, & \text{otherwise}.
  \end{cases}
\end{align*}
This proves the claim as $(1+(n-1))/n = 1$ and $(1 + 0)/n = n^{-1}$.
\end{proof}

\subsection*{Proof of \RefProp{prop:constant}}
For the reader's convenience let us first recall a standard estimate which is sometimes referred to as \emph{Jensen's inequality}.
\begin{lemma}\label{lem:jensens}
Let $(a_j)_{j\in\N}$ denote an arbitrary sequence of non-negative real numbers. Then, for every $0<q\leq p< \infty$, it holds
\begin{equation*}
	\left( \sum_{j=1}^\infty a_j^p \right)^{1/p}
	\leq \left( \sum_{j=1}^\infty a_j^q \right)^{1/q}
\end{equation*}
whenever the right-hand side is finite.
\end{lemma}
The proof of \RefProp{prop:constant} now reads as follows:
\begin{proof}
\emph{Step 1.}
To show the monotonicity of $C_{d,\lambda}(r_{\alpha,\bm{\beta}})$ we simply apply Jensen's inequality with the exponent $p=\mu/\lambda \geq 1=q$. The continuous dependence on $\lambda$ follows from the fact that $\ell_p$-sequence norms are continuous w.r.t.\ $p$ and the representation \link{eq:scaling} can easily be verified since the squared initial error on $\SI_{I_d}(F_d(r_{\alpha,\bm{\beta}}))$ is given by $\beta_0^d$; see \link{eq:M2Id}.

\emph{Step 2.}
For $\lambda=1$ the identities proven in \RefLem{lem:SdMd} together with \link{eq:transform} show that 
$C_{d,1}(r_{\alpha,\bm{\beta}})$ equals
\begin{align*}
	\sum_{\bm{0}\neq \bsh \in \Z^d}\frac{\M_d(\bsh)!}{\# S_d} \, r^{-1}_{\alpha,\bm{\beta}}(\bsh)
	= \sum_{\bsh \in \Z^d}\frac{\M_d(\bsh)!}{\# S_d} \, r^{-1}_{\alpha,\bm{\beta}}(\bsh) - \beta_0^d
	&= \sum_{\bsk \in \nabla_d} r^{-1}_{\alpha,\bm{\beta}}(\bsk) - \beta_0^d \\
	&= \beta_0^d \left( \frac{M_{2,d}(K_{d,I_d})}{S_d(K_{d,I_d})} - 1 \right) 
\end{align*}
which agrees with \link{eq:repres}. To prove the lower bound we note that for every $m\in\N$ the vector $\bm{m}=(m,\ldots,m)$ as well as its negative belong to the set $\nabla_d$ such that
\begin{equation*}
	\sum_{\bsk \in \nabla_d} r^{-1}_{\alpha,\bm{\beta}}(\bsk)
	\geq 2 \, r_{\alpha,\bm{\beta}}^{-1}(\bm{m}) = 2 \, \left[ \beta_1 \, R(m)^{-2\alpha} \right]^d.
\end{equation*}

\emph{Step 3.} The proof of the remaining lower bounds is based on the arguments already used in the proof of \autoref{thm:rmse}. There we defined sets of indices $\J_d\subset \Z^d \setminus \{\bm{0}\}$ whose elements behave well under permutations $P\in\S_d$. Using essentially the same calculations we obtain the bounds
\begin{equation*}
	C_{d,\lambda}(r_{\alpha,\bm{\beta}})
		\geq \beta_0^d \left(2 \sum_{\ell = 1}^{d} \binom{d}{\ell}^{1-1/\lambda} \left[ \frac{\beta_1}{\beta_0} \right]^{\ell/\lambda} \NR{\alpha\,\ell/\lambda} \right)^{\lambda}
\end{equation*}
if $\# I_d=d$, and 
\begin{align*}
	C_{d,\lambda}(r_{\alpha,\bm{\beta}})
	&\geq \beta_0^d \left[ \left( 1 + 2 \left[ \frac{\beta_1}{\beta_0} \right]^{1/\lambda} \NR{\alpha/\lambda} \right)^{d-\# I_d} - 1 \right]^{\lambda} \\
	&\qquad\qquad \qquad \times \left( 1 + 2 \sum_{\ell = 1}^{\# I_d} \binom{\# I_d}{\ell}^{1-1/\lambda} \!\left[ \frac{\beta_1}{\beta_0} \right]^{\ell/\lambda} \NR{\alpha\,\ell/\lambda} \right)^{\lambda}
	\end{align*}
if $\# I_d<d$ (where the second factor is not present for $\#I_d=0$). Note that both the last formulas hold true for general $\lambda \geq 1$ (but observe that $\NR{x}$ is infinite for $x\leq 1/2$) and that for every $\ell \in \N$ we have the lower estimate
\begin{equation*}
	\NR{\alpha \, \ell/\lambda} 
	= \sum_{m=1}^\infty R(m)^{-2\alpha \, \ell/\lambda} 
	\geq \left[ \frac{1}{R(1)^{2\alpha}} \right]^{\ell/\lambda}.
\end{equation*}
This proves the lower bounds for the case $\lambda=1$. 
Thus we are left with the case $\lambda>1$. Here the bound for $I_d=\emptyset$ is obvious. 
For $\# I_d>0$ we use Jensen's inequality (\RefLem{lem:jensens}) with $q=1-1/\lambda<1=p$ and the binomial theorem to obtain
\begin{align*}
	\sum_{\ell = 1}^{\# I_d} \binom{\# I_d}{\ell}^{1-1/\lambda}\! \left[ \frac{\beta_1}{\beta_0\, R(1)^{2\alpha}} \right]^{\ell/\lambda} \! 
	&= \left( \left[ \sum_{\ell = 1}^{\# I_d} \left( \binom{\# I_d}{\ell} \left[ \frac{\beta_1}{\beta_0\, R(1)^{2\alpha}} \right]^{\ell/(\lambda-1)} \right)^{1-1/\lambda} \right]^{1/(1-1/\lambda)} \right)^{1-1/\lambda} \\
	&\geq \left( \sum_{\ell = 1}^{\# I_d} \binom{\# I_d}{\ell} \left[ \frac{\beta_1}{\beta_0\, R(1)^{2\alpha}} \right]^{\ell/(\lambda-1)} \right)^{1-1/\lambda} \\
	&= \left( \left( 1 + \left[ \frac{\beta_1}{\beta_0\, R(1)^{2\alpha}} \right]^{1/(\lambda-1)} \right)^{\#I_d} -1 \right)^{1-1/\lambda}.
\end{align*}
Consequently, we have
\begin{align}
	&\left(2 \sum_{\ell = 1}^{\#I_d} \binom{\#I_d}{\ell}^{1-1/\lambda} \left[ \frac{\beta_1}{\beta_0} \right]^{\ell/\lambda} \NR{\alpha\,\ell/\lambda} \right)^{\lambda} 
	\geq 2^\lambda \left[ \left( 1 + \left[ \frac{\beta_1}{\beta_0\, R(1)^{2\alpha}} \right]^{1/(\lambda-1)} \right)^{\#I_d} -1 \right]^{\lambda-1} \label{est:bin_sum}
\end{align}
which yields the bound for $\#I_d=d$. For $0<\#I_d<d$ we apply Jensen's inequality once again (this time with $q=1<\lambda=p$) and derive
\begin{align*}
	&\left( 1 + 2 \sum_{\ell = 1}^{\# I_d} \binom{\# I_d}{\ell}^{1-1/\lambda} \left[ \frac{\beta_1}{\beta_0} \right]^{\ell/\lambda} \NR{\alpha\,\ell/\lambda} \right)^{\lambda} 
	\geq  1 + \left( 2 \sum_{\ell = 1}^{\# I_d} \binom{\# I_d}{\ell}^{1-1/\lambda} \left[ \frac{\beta_1}{\beta_0} \right]^{\ell/\lambda} \NR{\alpha\,\ell/\lambda} \right)^{\lambda}.
\end{align*}
The assertion now follows from \link{est:bin_sum}.

\emph{Step 4.} It remains to show the upper bound \link{est:upper_bound}. Therefore let $\lambda$, $A$, and $\gamma$ be given and note that the restrictions on the choice of $A$ are equivalent to
\begin{equation*}
	\alpha - A > \frac{1}{2}
	\qquad \text{and} \qquad
	\frac{A}{\lambda - 1} > 1/2.
\end{equation*}
Hence the quantities $\NR{\alpha-A}$ (which appears in $C_{d,1}(r_{\alpha-A,(\beta_0, \beta_1 \gamma)})$), as well as $\NR{A/(\lambda-1)}$ (which appears in the other factor), are finite. Applying H\"older's inequality yields
\begin{align*}
	C_{d,\lambda}(r_{\alpha,(\beta_0,\beta_1)})^{1/\lambda} 
	&= \sum_{\bm{0}\neq \bsh \in \Z^d} \left[ \frac{\M_d(\bsh)!}{\# \S_d} \, r^{-1}_{\alpha,(\beta_0,\beta_1)}(\bsh) \, r_{A,(1,1/\gamma)}(\bsh) \right]^{1/\lambda} r^{-1/\lambda}_{A,(1,1/\gamma)}(\bsh) \\
	&\leq \left( \sum_{\bm{0}\neq \bsh \in \Z^d} \frac{\M_d(\bsh)!}{\# \S_d} \, r^{-1}_{\alpha-A,(\beta_0,\beta_1\gamma)}(\bsh) \right)^{1/\lambda} \left( \sum_{\bm{0}\neq \bsh \in \Z^d} r^{-1/(\lambda-1)}_{A,(1,1/\gamma)}(\bsh) \right)^{1-1/\lambda},
\end{align*}
since $r^{-1}_{\alpha,(\beta_0,\beta_1)}(\bsh) \, r_{A,(1,1/\gamma)}(\bsh) = r^{-1}_{\alpha-A,(\beta_0,\beta_1\gamma)}(\bsh)$ for every $\bsh\in\Z^d$. Now the first sum obviously equals $C_{d,\lambda}(r_{\alpha-A,(\beta_0,\beta_1\gamma)})$, whereas the second sum can be calculated in the usual way using the tensor product structure of $r$:
\begin{align*}
	\sum_{\bm{0}\neq \bsh \in \Z^d} r^{-1/(\lambda-1)}_{A,(1,1/\gamma)}(\bsh) 
	&= \sum_{h_1 \in\Z}\cdots \sum_{h_d \in \Z} \prod_{\ell=1}^d r^{-1/(\lambda-1)}_{A,(1,1/\gamma)}(h_\ell) - 1 \\
	&= \left[ 1 + 2 \, \gamma^{-1/(\lambda-1)} \, \NR{A/(\lambda-1)} \right]^d - 1.
\end{align*}
This completes the proof.
\end{proof}

\section*{Acknowledgments}
\addcontentsline{toc}{section}{Acknowledgments}
This research is part of a project funded by the Research Fund KU Leuven. The third author has been supported by Deutsche Forschungsgemeinschaft DFG (DA~360/19-1).
In addition, the authors are grateful to the two anonymous reviewers who helped to improve the paper.

\addcontentsline{toc}{chapter}{References}
\bibliographystyle{spmpsci}      


\end{document}